\documentclass[12pt,letterpaper]{amsart}
\usepackage{amsfonts}
\usepackage{}
\usepackage{amssymb}
\usepackage{amsmath,amssymb,amsbsy,amsfonts,amsthm,latexsym,
                        amsopn,amstext,amsxtra,euscript,amscd,mathrsfs,color,bm,cite}
\usepackage{float}
\usepackage[english]{babel}
\usepackage{mathtools}
\usepackage{todonotes}
\usepackage{url}
\usepackage[colorlinks,linkcolor=blue,anchorcolor=blue,citecolor=blue,backref=page]{hyperref}
\usepackage{cases}
\usepackage{txfonts}
\usepackage{latexsym}
\usepackage{amsthm}
\usepackage{mathrsfs}
\usepackage{amssymb, amsmath}
\usepackage{cite}
\usepackage{color}
\usepackage{verbatim}

\def\re{\text{Re}}
\def\ma{\mathfrak{a}}
\def\mb{\mathfrak{b}}
\def\mc{\mathfrak{c}}
\def\mA{\mathcal{A}}
\def\mB{\mathscr{B}}

\def\mP{\mathcal{P}}

\def\mG{\mathcal{G}}
\def\mH{\mathcal{H}}
\def\mM{\mathscr{M}}
\def\mmM{\mathcal{M}}
\def\mE{\mathcal{E}}
\def\mX{\mathcal{X}}
\def\mQ{\mathcal{Q}}
\def\mV{\mathcal{V}}
\def\r){\right)}
\def\B{\Bigg}

\def\d{\mathrm{d}}
\def\r{\right}
\newcommand{\bC}{\mathbb{C}}

\def\msum{\mathop{\sum\nolimits^+}}

\def\ssum{\mathop{\sum\nolimits^*}}

\def\abgd{\alpha,\beta,\gamma,\delta}
\def\bagd{\beta,\alpha,\gamma,\delta}
\def\abdg{\alpha,\beta,\delta,\gamma}
\def\badg{\beta,\alpha,\delta,\gamma}

\def\gdab{\gamma,\delta,\alpha,\beta}

\def\dgba{\delta,\gamma,\beta,\alpha}

\def\ppmod{\!\!\!\!\!\pmod}
\def\ssqrt{\!\sqrt}

\let\ve=\varepsilon

\let\ol=\overline

\let\vp=\varphi
\let\wt=\widetilde
\let\wh=\widehat

\newcommand{\sumstar}{\sideset{}{^*}\sum}
\newcommand{\shortmod}{\ensuremath{\negthickspace \negthickspace \negthickspace \pmod}}
\def\phis{\varphi^*}

\theoremstyle{definition}

\newtheorem*{remark}{Remark}

\theoremstyle{plain}
\newtheorem{theorem}{Theorem}

\newtheorem{lemma}[theorem]{Lemma}

\numberwithin{equation}{section}
\numberwithin{theorem}{section}
\newtheorem{proposition}[theorem]{Proposition}
\def\qed{\ifhmode\textqed\fi
   \ifmmode\ifinner\quad\qedsymbol\else\dispqed\fi\fi}
\def\textqed{\unskip\nobreak\penalty50
    \hskip2em\hbox{}\nobreak\hfil\qedsymbol
    \parfillskip=0pt \finalhyphendemerits=0}
\def\dispqed{\rlap{\qquad\qedsymbol}}

\begin{document}
\title{The mollified fourth moment of Dirichlet $L$-functions}
\author[P. Gao]{Peng Gao}
\address{School of Mathematical Sciences, Beihang University, Beijing 100191, China}
\email{penggao@buaa.edu.cn}

\author[X. Wu]{Xiaosheng Wu}
\address {School of Mathematics, Hefei University of Technology, Hefei 230009,
China}
\email {xswu@amss.ac.cn}

\author[L. Zhao]{Liangyi Zhao}
\address{School of Mathematics and Statistics, University of New South Wales, Sydney NSW 2052, Australia}
\email{l.zhao@unsw.edu.au}

\begin{abstract}
We prove an asymptotic formula with a power saving error term for the fourth moment of the family of Dirichlet $L$-functions to modulus $q$ mollified by a Dirichlet polynomial of length $q^{\frac1{22}-\ve}$, valid for all moduli $q\not\equiv2 \pmod 4$.  This result was previously known only for restricted sets of moduli with smaller power savings.  As a special case, when no Dirichlet polynomial is enclosed, this leads to a significant improvement on X. Wu's asymptotic evaluation of the fourth moment of Dirichlet $L$-functions at the cental point.
\end{abstract}


\maketitle

\section{Introduction}

  It is an important subject to study moments of families of $L$-functions in number theory as they have a wide range of applications. In this paper, we focus on the classical setting of moments of the family of Dirichlet $L$-functions to a fixed modulus. Based on the work of J. B. Conrey, D. W. Farmer, J. P. Keating, M. O. Rubinstein and N. C. Snaith in \cite{CFK+05}, it is conjectured that
  \begin{align}
\label{moments1}
 M(k,q):=\sumstar_{\substack{ \chi \shortmod q }}|L(\tfrac{1}{2},\chi)|^{2k} =\phis(q)P_{k^2}(\log q)+O\left(q^{-\frac12+\ve}\right),
\end{align}
where $\sum^*$ indicates a sum over primitive Dirichlet characters $\chi$ modulo $q$,  $L(s, \chi)$ the $L$-function attached to $\chi$ and $\phis(q)$ the number of such primitive characters. We henceforth assume that $q \not \equiv 2 \pmod 4$ to ensure the existence of such primitive $\chi$. Here, $P_{k^2}(x)$ is an explicitly computable polynomial
\[
P_{k^2}(x)=a_{k^2}x^{k^2}+a_{k^2-1}x^{k^2-1}+\cdots +a_0
\]
with coefficients in terms of finite Euler product
\[
a_j=C_j\prod_{p\mid q}(1+f_j(p))\ll (\log\log q)^{O(1)}
\]
for explicit constants $C_j$ and arithmetic functions $f_j(p)\ll p^{-1}$. Throughout the paper, we reserve the letter $p$ for a prime number. The explicit expressions of $a_j$ can be calculated from the shifted moments by setting all shifts to zero (c.f Theorem \ref{thmmain1} for the case $k=2$).
In particular, when $q$ is a prime, all coefficients $a_j$ are absolute constants. In particular, \eqref{moments1} implies that asymptotically,
\begin{equation}
\label{moments}
 M(k,q)\sim a_{k^2} \phis(q)(\log q)^{k^2},
\end{equation}
which is also expected (see \cite{RS05}) to hold for all real $k  \geq 0$.

     The validity of \eqref{moments} for $k=1$ follows from a result of A. Selberg \cite{Selberg46} on a more general twisted second moment. Asymptotic evaluation \eqref{moments} for $k=2$ was done by D. R. Heath-Brown \cite{HB81} for most all $q$, whose result was subsequently extended to all $q$ by K. Soundararajan \cite{Sou07}. For prime moduli $q$, M. P. Young \cite{You11} was the first to obtain the asymptotic formula \eqref{moments1} for $M(k,q)$ when $k=2$ with a power saving error term $q^{-\frac5{512}+\ve}$. The error term in Young's result are subsequently improved in \cite{BFK+17a, BFK+17b,BFKMMS23} to $q^{-\frac1{20}+\ve}$, which is currently the best known result for prime moduli. For general moduli $q$, the asymptotic formula \eqref{moments1} was later obtained by X. Wu \cite{Wu23} with a power saving error term $q^{-\frac{11}{448}+\ve}$.

 Instead of studying $M(k,q)$, one may consider more generally the so-called twisted shifted moments. For example, we say the corresponding case for $k=2$ is the twisted shifted fourth moment, which has the form
\begin{equation}
\label{twistedshiftedfourthmoments}
 \sumstar_{\substack{ \chi \shortmod q }} L\left(\tfrac{1}{2} + \alpha, \chi \right)L\left(\tfrac{1}{2} + \beta, \chi \right)L\left(\tfrac{1}{2} + \gamma, \overline{\chi} \right)L\left(\tfrac{1}{2} + \delta, \overline{\chi} \right)\chi(h)\overline{\chi}(k),
\end{equation}
where $\alpha$, $\beta$, $\gamma$ and $\delta$ are complex numbers, and $h, k \in \mathbb{Z}$.

 For prime moduli $q$,  the expression in \eqref{twistedshiftedfourthmoments} was asymptotically evaluated for square-free $h,k$ by B. Hough \cite[Theorem 4]{Hough16}. The result was extended to cube-free $h, k$ by R. Zacharias \cite{Zach19}, and to general $h,k$ by D. Liu \cite[Theorem 6.2]{Liu24}. In \cite{GZ25+b}, P. Gao and L. Zhao obtained an asymptotical formula for the expression in \eqref{twistedshiftedfourthmoments} for $q$ being certain prime powers and general $h, k$.

   Evaluation of the twisted shifted moments is essential towards our understanding of the so-called mollified moments, which, for $k=2$, has the form
\begin{equation}
\label{mollifiedmoment}
\sumstar_{\chi\ppmod q} L\left(\tfrac12+\alpha,\chi\r)L\left(\tfrac12+\beta,\chi\r) L\left(\tfrac12+\gamma,\ol{\chi}\r)L\left(\tfrac12+\delta,\ol{\chi}\r)\left|A\left(\tfrac12,\chi\r)\right|^2,
\end{equation}
   where $A\left(\tfrac12,\chi\r)$ is a Dirichlet polynomial
\begin{equation}
\label{Adef}
A\left(\tfrac12,\chi\r)=\sum_{h\le y}\frac{\alpha_h\chi(h)}{\sqrt{h}}
\end{equation}
with the coefficients satisfying $\alpha_h\ll h^\ve$.

The mollified moment plays a pivotal role in an extensive range of problems, and asymptotic formulas with a power saving error term are essential prerequisites for many applications including the techniques of amplification, modification or resonance. For example, they have been applied to study extreme values of $L$-functions \cite{Sou08}, non-vanishing of $L$-functions at the central point \cite{CS07}, and to achieve sharp bounds for moments of $L$-functions with real exponents \cite{RadSou15}.

  The degenerate case when $h=k=1$ in \eqref{twistedshiftedfourthmoments} is called the shifted fourth moment. One may define shifted moments for other families of $L$-functions as well.  As revealed in \cite{Munsch17, Szabo24}, these shifted moments are especially useful in estimating moments of the corresponding character sums.

   In view of the importance of the twisted shifted moments and mollified moments, it is the aim of this paper to evaluate asymptotically the expressions in \eqref{twistedshiftedfourthmoments} and \eqref{mollifiedmoment} for general moduli $q$. As the situations are similar, we consider only the case of even characters (i.e. those characters $\chi$ with $\chi(-1)=1$) throughout the paper. For this, we define
\begin{align*}
\mmM_{h,k}&(\alpha,\beta,\gamma,\delta):=\\
&\frac2{\varphi^*(q)}\msum_{\chi\ppmod q} L\left(\tfrac12+\alpha,\chi\r)L\left(\tfrac12+\beta,\chi\r) L\left(\tfrac12+\gamma,\ol{\chi}\r)L\left(\tfrac12+\delta,\ol{\chi}\r)\chi(h\ol k),
\end{align*}
\begin{align*}
\mmM&(\alpha,\beta,\gamma,\delta):=\\
&\frac2{\varphi^*(q)}\msum_{\chi\ppmod q} L\left(\tfrac12+\alpha,\chi\r)L\left(\tfrac12+\beta,\chi\r) L\left(\tfrac12+\gamma,\ol{\chi}\r)L\left(\tfrac12+\delta,\ol{\chi}\r)\left|A\left(\tfrac12,\chi\r)\right|^2,
\end{align*}
 where $\ol k$ denotes the inverse of $k$ modulo $q$, the symbol $\sum^+$ henceforth indicates that the summation is over all primitive even characters, and $A$ is defined in \eqref{Adef}.

  To state our result, we need to introduce some notations. Let
\begin{equation*}
X_{\alpha,\gamma}=X_{\alpha}X_{\gamma},\ \ \ \ \ \ \ \ X_{\alpha,\beta,\gamma,\delta}=X_{\alpha}X_{\beta}X_{\gamma}X_{\delta}
\end{equation*}
with
\begin{equation*}
X_{\alpha}=\left(\frac{q}{\pi}\r)^{-\alpha}\frac{\Gamma\left(\frac{\frac12-\alpha}{2}\r)} {\Gamma\left(\frac{\frac12+\alpha}{2}\r)}.
\end{equation*}
Set
\begin{equation}
\label{Zdef}
Z_{h,k,q}(\abgd)=Y_h(\abgd)Y_{k}(\gdab)Z_q(\abgd),
\end{equation}
where
\[
Z_q(\abgd)=\frac{\zeta_q(1+\alpha+\gamma)\zeta_q(1+\alpha+\delta)\zeta_q(1+\beta+\gamma)\zeta_q(1+\beta+\delta)}{\zeta_q(2+\alpha+\beta+\gamma+\delta)}
\]
with $\zeta_{q}(s)$ the Euler product of $\zeta(s)$ with the primes dividing $q$ removed and
\begin{equation}\label{eqdefY}
Y_{a}(\abgd)=\frac1{a^{\gamma}}\sum_{d\mid a}d^{\gamma-\delta}\prod_{p\mid d}\left(1-\frac1{p^{1+\alpha+\gamma}}\r) \left(1-\frac1{p^{1+\beta+\gamma}}\r) \left(1-\frac1{p^{2+\alpha+\beta+\gamma+\delta}}\r)^{-1}.
\end{equation}
It is easy to see that $Z_q(\abgd)$ is invariant under transpositions $\alpha\leftrightarrow\beta$ and $\gamma\leftrightarrow\delta$, and so is $Y_a(\abgd)$ (see Lemma \ref{lemYa} below). As a result, this invariance under transpositions $\alpha\leftrightarrow\beta$ and $\gamma\leftrightarrow\delta$ also hold for $Z_{h,k,q}(\abgd)$.

\subsection{Main results}
  Now, we are ready to state our results on $\mmM_{h,k}(\alpha,\beta,\gamma,\delta)$ and $\mmM(\abgd)$. We begin with $\mmM_{h,k}(\alpha,\beta,\gamma,\delta)$.
\begin{theorem}\label{thmmain1}
 With the notation as above, let $h,k$ be integers satisfying $(h,k)=(hk,q)=1$. Then there exists $\eta > 0$ such that for $\alpha, \beta, \gamma, \delta \in \left\{z \in \bC: \Re(z) < \eta/\log q \right\}$, we have
\begin{equation}
\label{Mhkasmp}
\mmM_{h,k}(\alpha,\beta,\gamma,\delta)=(hk)^{-\frac12}\mathfrak{M}_{h,k}(\abgd)+O\left(q^{-\frac1{20}+\ve}(|h|+|k|)^{\frac{3}{10}}\Delta^{O(1)}\r),
\end{equation}
where
\begin{align*}
\Delta=& (1+|\alpha|)(1+|\beta|)(1+|\gamma|)(1+|\delta|), \; \mbox{and} \\
\mathfrak{M}_{h,k}(\abgd)=Z_{h,k,q} & (\abgd)+X_{\abgd}Z_{h,k,q}(-\gamma,-\delta,-\alpha,-\beta) \\
&+X_{\alpha,\gamma}Z_{h,k,q}(\beta,-\gamma,\delta, -\alpha)  +X_{\beta,\gamma}Z_{h,k,q}(\alpha,-\gamma,\delta, -\beta)\\
&+X_{\alpha,\delta}Z_{h,k,q}(\beta,-\delta,\gamma, -\alpha)  +X_{\beta,\delta}Z_{h,k,q}(\alpha,-\delta,\gamma, -\beta).
\end{align*}
\end{theorem}

  Our next result evaluates $\mmM(\abgd)$.
\begin{theorem}\label{thmmain}
 With the notation as above, there exists $\eta > 0$ such that for $\alpha, \beta, \gamma, \delta \in \left\{z \in \bC: \Re(z) < \eta/\log q \right\}$, we have
\begin{equation*}
\mmM(\alpha,\beta,\gamma,\delta)=\sum_{\substack{ah,ak\le y\\ (ahk,q)=1\\ (h,k)=1}}\frac{\alpha_{ah}\ol{\alpha_{ak}}}{ahk}\mathfrak{M}_{h,k}(\abgd)+O\left(q^{-\frac1{20}+\ve}y^{\frac{11}{10}} \Delta^{O(1)}\r).
\end{equation*}
\end{theorem}

\begin{remark}

To facilitate the reader's understanding of the strength of Theorems \ref{thmmain1}--\ref{thmmain}, we make some remarks here.
\begin{itemize}
\item If one sets $h=k=1$ and all shifts to zero, then Theorem \ref{thmmain1} establishes an asymptotic formula for the fourth moment of Dirichlet $L$-functions as in \cite{Wu23}, but with the error term being sharpened to $q^{-\frac1{20}+\ve}$. This error term is as strong as the best known result established for the prime moduli case in \cite{BFK+17a, BFK+17b,BFKMMS23}. Given that our method of treating the error terms is quite different from the those carried out in \cite{BFK+17a, BFK+17b,BFKMMS23} concerning the prime moduli case, curiosity naturally arises regarding the reason for this identical end via vastly distinct means.
    We note here that upon careful comparison, the term $q^{-\frac1{20}+\ve}$ in both situations ultimately stems from certain diagonal terms introduced when applying the Cauchy-Schwarz inequality. This suggests, to a certain extent, that any further improvement in the $q$-aspect seems very unlikely via these approaches.
\item For twisted moments, we contrast Theorem \ref{thmmain1} with previously established results for various special cases, such as \cite{Hough16, Zach19, Liu24, GZ25+b}. We note that Theorem \ref{thmmain1} establishes the asymptotic formula for general moduli $q$ and general $h, k$, sharping the error term in all aspects including $q$, $h$, and $k$.
\item The asymptotic formula of the mollified moment in Theorem \ref{thmmain} is valid (the main term dominates the $O$-term) if the length of the mollifier does not exceed $q^{\frac1{22}-\ve}$.
    This situation is somewhat analogous to the mollified fourth moment of the Riemann zeta-function, established in the works of C. P. Hughes and P. Y. Matthew \cite{HM10} and S. Bettin, H. M. Bui, X. Li, and M. Radziwi\l{}\l{} \cite{BBLR16}, where stationary phase methods applied to the Archimedean integral over $t$ impose stringent constraints on the shifted convolution sum through the variable clustering effect, thereby producing significantly improved error terms.
\item If one only aims for an upper bound, then we note that
V. Blomer, P. Humphries, R. Khan, and M. B. Milinovich \cite{BHKM20} applied their Motohashi's fourth moment identity to obtain that
\[
M(0,0,0,0)\ll_\ve q^{\ve},
\]
for prime moduli $q$ and square-free sequences $\alpha_h\ll h^\ve$. Their upper bound holds for a very long mollifier with $y\ll q^{\frac14-\ve}$. However, as mentioned in \cite{BHKM20}, their method does not give asymptotic results.
\item In view of Theorem \ref{thmmain1}, Theorem \ref{thmmain} gains some additional power saving in the $y$-aspect through averaging over $h$ and $k$.
\item The utility and relevance of Theorems \ref{thmmain1}--\ref{thmmain} clearly lie in the study of the family of Dirichlet $L$-functions (more specifically, non-vanishing at $s=1/2$ and existence of large values, for instance).
    As a direct application, one may apply the approach in the work of Gao and Zhao \cite{GZ25+a}, while using the result of Theorem \ref{thmmain} in place of those in \cite{Hough16} or  \cite{Zach19}. This allows one to extend the result \cite{GZ25+a} to establish sharp upper bounds for all $2k$-th moment for $k \leq 2$ at the central point of the family of Dirichlet $L$-functions to general moduli.
\end{itemize}
\end{remark}

 Our proofs of Theorems \ref{thmmain1}--\ref{thmmain} loosely follow the {\it modus operandi} of \cite{You11} and \cite{Wu23}. The calculation of the main terms is quite involved, and their discernment, most crucially, comes from studying a smoothed version of the generalized Esterman $D$-functions (discussed in Section \ref{sectionD}). The treatments for the error terms are divided into two parts: the balanced ones and the unbalanced ones. The balanced error terms are handled using spectral theory and the spectral large sieve, where we remove the dependence of the Ramanujan--Petersson conjecture and minimize the impact from $h, k$ by establishing new upper bounds for sums involving Kloosterman sums, Theorems \ref{thKls}--\ref{thKls1}. The treatment of the unbalanced error terms is based on applying new bound for bilinear forms of incompleted Kloosterman sums, which was proved for general moduli by B. Kerr, I. E. Shparlinski, X. Wu, and P. Xi \cite{KSWX23}.

  We note that as our primary goal is to obtain power saving error terms in the $q$ aspect when evaluating $\mmM_{h,k}(\alpha,\beta,\gamma,\delta)$ and $\mmM(\alpha,\beta,\gamma,\delta)$.  Thus we do not make explicit the dependence on the parameters $\alpha, \beta, \gamma, \delta$ in the error terms. Often, in our treatments below, we shall omit the dependency on these parameters as well.

\subsection{Upper bounds for sums of Kloosterman sums}

  In our proof of Theorems \ref{thmmain1}--\ref{thmmain}, we need to estimate certain shifted convolution sums that appear in the process (see Section \ref{secE}), which requires a good control on related sums involving with Kloosterman sums. We develop the necessary tools in this paper as well.

  Let $\ma$ and $\mb$ denote two cusps of the Hecke congruence group $\Gamma_0(Q)$, with corresponding scaling matrixes $\sigma_{\ma}$, $\sigma_{\mb}$. For integers $m,n$ and real number $\gammaup$ such that there exists a matrix
$\begin{pmatrix}
     \alphaup & \betaup \\
     \gammaup & \deltaup
\end{pmatrix}\in \sigma_{\ma}^{-1}\Gamma_0(Q)\sigma_{\mb}$, the associated Kloosterman sum $S_{\ma,\mb}(m,n;\gamma)$ is defined as
\begin{equation}\label{eqdefS}
S_{\ma,\mb}(m,n;\gamma)=\mathop{\sum\nolimits^{'}}_{\deltaup\ppmod {\gammaup\mathbb{Z}}}e\left(m\frac{\alphaup}{\gammaup}+n\frac{\deltaup}{\gammaup}\right),
\end{equation}
where the sum runs over all $\deltaup$ modulo $\gamma\mathbb{Z}$ for which there exist $\alpha$, $\beta$ completing the matrix $\begin{pmatrix}
     \alphaup & \betaup \\
     \gammaup & \deltaup
\end{pmatrix}$
in $\sigma_{\ma}^{-1}\Gamma_0(Q)\sigma_{\mb}$.
When $\ma=\mb=\infty$ and $Q=q$, this reduces to
the classical Kloosterman sum $S(m,n;q)$.

Sums of the type
\begin{equation}\label{eqsks}
\sum_{M<m\le 2M}a_m\sum_{N<n\le 2N} b_n \sum_{\gammaup}\frac1{\gammaup}\phi\left(\frac{4\pi\ssqrt{mn}}{\gammaup}\r)S_{\ma,\mb}(m,\pm n; \gammaup)
\end{equation}
have been widely studied due to their numerous applications.  Here $\phi$ is a smooth test function satisfying
\begin{equation}\label{eqg1}
\text{Supp} \phi\subset[X,8X],
\end{equation}
and
\begin{equation}\label{eqg2}
\|\phi\|_{\infty}\ll 1,\ \ \ \|\phi'\|_{1}\ll 1,\ \ \ \ \|\phi''\|_{1}\ll X^{-1},
\end{equation}
for some positive parameter $X$.  The classical upper bound for such sums was established by J. M. Deshouillers and H. Iwaniec \cite[Theorem 8]{DI82}, using spectral analysis.

The asymptotic twisted moment problem requires evaluating the shifted convolution sum
\[
\mathop{\sum\sum}_{\substack{hm\equiv kn \ppmod{q}\\ (mn,q)=1,\ hm\neq kn}} \tau(m)\tau(n)W\left(\frac mM\right)W\left(\frac nN\right),
\]
for $h, k, M, N\ge1$ and smooth test functions $W$.
The error terms in such evaluation depend critically on estimating sums of the form \eqref{eqsks} involving Kloosterman sums $S_{\infty,\mb}(s m,\pm n; \gammaup)$ with an individual shift $s$ (typically large). Here, the level of the Hecke congruence group $Q$ is roughly a product of $hk$ with certain factors of $q$ arising from the comprimality condition $(mn,q)=1$. Usually, the spectral analysis of such Kloosterman sums with an individual shift will produce error terms involving $s$-power factors, whose exponents depend on the current state of our knowledge towards the Ramanujan--Petersson conjecture.
When establishing an asymptotic for the second moment of twisted modular $L$-functions, one must handle a similar shifted convolution sum
\[
\mathop{\sum\sum}_{\substack{m\equiv n \ppmod{q}\\ (mn,q)=1,\ m\neq n}} \lambda_f(m)\lambda_g(n)W\left(\frac mM\right)W\left(\frac nN\right),
\]
where $f$ and $g$ denote fixed cuspidal Hecke eigenforms, with respective Hecke eigenvalues $\lambda_f$ and $\lambda_g$.
For this sum,
Blomer and Mili\'cevi\'c \cite{BM15} successfully eliminated the dependence on the Ramanujan--Petersson conjecture through a clever arrangement of H\"{o}lder's inequality after the spectral decomposition. We are able to adapt their idea to our present setting, incorporating new techniques to manage the higher level $Q$'s.
The key input is Theorem \ref{thmsls}, where we minimize the influence of $s$ through a novel relation of Fourier coefficients, Proposition \ref{prorho}.


Let $\lambda_1(Q)$ denote the least positive eigenvalue of the Laplacian for $\Gamma_0(Q)$, and we define the number
\[
\vartheta_Q=\sqrt{\max\{0, 1-4\lambda_1(Q)\}}.
\]
Selberg's eigenvalue conjecture posits that $\lambda_1(Q)\ge \frac14$, which directly implies $\vartheta_Q=0$. To date, the most comprehensive understanding we have regarding $\vartheta_Q$ in general is that $\vartheta_Q\le \frac12$.
Our results on Kloosterman sums are presented in the following two theorems.  These results can be of independent interest since they may be applied not only to analogous shifted convolution sums but also to other problems in analytic number theory.
\begin{theorem}\label{thKls}
We keep the notation as above. Let $s\in \mathbb{N}$, $Q=uv$ with $(u,v)=1$, $M, N\ge 1$, and set $C=M+N+Q+X+s$.  Suppose that $\phi$ is a real valued function  satisfying \eqref{eqg1} and \eqref{eqg2}.  For any complex sequence $(a_m)_{m\in[M,2M]}, (b_n)_{n\in[N,2N]}$, we have
\begin{align}\label{eqKls}
\sum_{m}a_m\sum_{n} b_n &\sum_{\gammaup}\frac1{\gammaup}\phi\left(\frac{4\pi\ssqrt{smn}}{\gammaup}\r)S_{\infty,1/u}(s m,\pm n; \gammaup)\ll C^{\ve}\frac{1+|\log X|+X^{-\vartheta_Q}}{1+X}\\
&\times\B(1+X+\frac{s^{\frac12}}{Q^{\frac12}}\B)^{\frac12} \B(1+X+\frac{(s,Q)^{\frac12}M}{Q^{\frac12}}\B)^{\frac12}\B(1+X+\frac{N^{\frac12}}{Q^{\frac12}}\B)M^{\frac12}\|a_m\|_\infty^{\frac12}\|b_n\|_2.\notag
\end{align}
\end{theorem}

Through standard Fourier analysis (for instance, as in the proof for \cite[Theorem 13, p.269]{DI82}), one can derive from Theorem \ref{thKls} the subsequent result.
\begin{theorem}\label{thKls1}
Keeping the notation as in Theorem~\ref{thKls}, suppose that $g(m,n,l)$ is a real valued function of $\mathscr{C}^6$ class, having compact support in $[M,2M]\times[N,2N]\times[L,2L]$ and satisfying
\[
l^{j_1}m^{j_2}n^{j_3}\frac{\partial^{j_1+j_2+j_3}}{\partial l^{j_1}\partial m^{j_2}\partial n^{j_3}}g(m,n,l)\ll C^\ve \quad \text{with}\ \ 0\le j_1, j_2, j_3\le 2.
\]
For any complex sequences $a_m, b_n$, we have
\begin{align*}
\sum_{(l,v)=1}\sum_{m}a_m&\sum_{n} b_n g(m,n,l)S(s m \ol{v},\pm n; lu)\ll C^{\ve}\frac{1+|\log X|+X^{-\vartheta_Q}}{1+X}uv^{\frac12}L\\
&\times\B(1+X+\frac{s^{\frac12}}{Q^{\frac12}}\B)^{\frac12} \B(1+X+\frac{(s,Q)^{\frac12}M}{Q^{\frac12}}\B)^{\frac12}\B(1+X+\frac{N^{\frac12}}{Q^{\frac12}}\B)M^{\frac12}\|a_m\|_\infty^{\frac12}\|b_n\|_2,\notag
\end{align*}
where
\[
X=\frac{(sMN)^{\frac12}}{uv^{\frac12}L}.
\]
\end{theorem}

\subsection{Notation}
  As usual, we use $\ve$ to denote an arbitrarily small positive constant that may vary from place to place. We write $W(x)$ for a smooth non-negative function, which may have different expressions for each occurrence, but always has compact support in $[1/2,2]$ and satisfies
\begin{equation}\label{eqdefW}
W^{(j)}(x)\ll q^{j\ve}
\end{equation}
for any $j\ge0$.
We denote its Mellin transform by $\wt{W}$ and its Fourier transform by $\wh{W}$, with similar notation for other functions.
For $d, q\in \mathbb{N}$, denote by $q_d$ the maximal factor of $q$ such that $(q_d,d)=1$.

The notations $\sigma_\lambda(n)$ and $\sigma_{\alpha,\beta}(n)$ are defined as
\[
\sigma_\lambda(n)=\sum_{d\mid n}d^\lambda,\ \ \ \ \sigma_{\alpha,\beta}(n)=\sum_{ad=n}a^{-\alpha} d^{-\beta}.
\]
Clearly, $\sigma_{\alpha,\beta}(n)=n^{-\alpha}\sigma_{\alpha-\beta}(n)$. We also write $\tau(n)$ for the divisor function.

%
%
%
%
%
%
%

\section{Auxiliary lemmas}
\subsection{Approximate functional equation}
We shall evaluate $M_{h,k}(\alpha,\beta,\gamma,\delta)$  by first expanding it into a Dirichlet $L$-series,  using the following approximate functional equation, derived from the standard functional equation of $L(s,\chi)$. This result is taken from \cite[Proposition 2.4]{You11}.
\begin{lemma}[Approximate functional equation]
\label{lemafe}
Let $G(s)$ be an even, entire function of exponential decay in any strip $|\re(s)|<C$ satisfying
\[
G(0)=1,\quad G\left(\tfrac12\pm\alpha\r)=G\left(\tfrac12\pm\beta\r)=G\left(\tfrac12\pm\gamma\r)=G\left(\tfrac12\pm\delta\r)=0.
\]
Define
\[
V_{\alpha,\beta,\gamma,\delta}(x)=\frac1{2\pi i}\int_{(1)}\frac{G(s)}{s}g_{\alpha,\beta,\gamma,\delta}(s)x^{-s}ds,
\]
where
\begin{equation*}
g_{\alpha,\beta,\gamma,\delta}(s)=\pi^{-2s}\frac{\Gamma\left(\frac{ \frac12+\alpha+s+\ma}2\r) \Gamma\left(\frac{ \frac12+\beta+s+\ma}2\r) \Gamma\left(\frac{ \frac12+\gamma+s+\ma}2\r) \Gamma\left(\frac{ \frac12+\delta+s+\ma}2\r)}{\Gamma\left(\frac{\frac12+\alpha+\ma}2\r) \Gamma\left(\frac{\frac12+\beta+\ma}2\r) \Gamma\left(\frac{\frac12+\gamma+\ma}2\r) \Gamma\left(\frac{\frac12+\delta+\ma}2\r)}, \ \ \ \ma=\frac {1-\chi(-1)}{2}.
\end{equation*}
We have
\begin{align}
L\left(\tfrac12+\alpha,\chi\r) &L\left(\tfrac12+\beta,\chi\r) L\left(\tfrac12+\gamma,\ol{\chi}\r) L\left(\tfrac12+\delta,\ol{\chi}\r)\notag\\
=&\sum_{m,n}\frac{\sigma_{\alpha,\beta}(m)\sigma_{\gamma,\delta}(n) \chi(m)\ol{\chi}(n)} {(mn)^{\frac12}}V_{\alpha,\beta,\gamma,\delta}\left(\frac{mn}{q^2}\r)\notag\\
& \hspace*{1cm} + X_{\abgd}\sum_{m,n}\frac{\sigma_{-\alpha,-\beta}(m)\sigma_{-\gamma,-\delta}(v) \ol{\chi}(m)\chi(n)} {(mn)^{\frac12}}V_{-\alpha,-\beta,-\gamma,-\delta}\left(\frac{mn}{q^2}\r).\notag
\end{align}
\end{lemma}

\subsection{The orthogonality formula}
After using the approximate functional equation to expand $M_{h,k}(\alpha,\beta,\gamma,\delta)$, we need to compute various sums over characters via a standard orthogonality formula presented below, available in \cite{HB81} and \cite{Sou07}.
\begin{lemma}[The orthogonality formula]
\label{lemof}
For $(mn,q)=1$, we have
\[
\mathop{{\sum}^+}_{\chi(\bmod q)}\chi(m)\ol{\chi}(n) =\tfrac12\sum_{d\mid(q,m\pm n)}\vp(d)\mu\left(\frac qd\r) .
\]
\end{lemma}

\subsection{Two partitions of unity}
  We also want to localize the variables in certain sums appearing in the proof of Theorems \ref{thmmain1}--\ref{thmmain}.  To that end, we utilize two smooth partitions of unity. The first is the classical dyadic partition.
\begin{lemma}[Smooth dyadic partition of unity]\label{lemdyadic}
There exists a smooth non-negative function $W(x)$ having a support on $[1/2,2]$ and satisfying \eqref{eqdefW} such that
\begin{equation}\label{eqdyadic}
\sum_{k\ge0}W\left(\frac{x}{2^k}\r)=1,
\end{equation}
for any $x\ge1$.
\end{lemma}
Our next lemma introduces a smooth partition designed to distinguish the relative sizes of two variables. This partition is also discussed in \cite{BBLR16} and \cite[Section 2.6]{Wu23+}.
\begin{lemma}\label{lempartion1}
There is a smooth function $\omega(x)$ such that
\[
\omega(x)+\omega\left(x^{-1}\r)=1,
\]
for any $x\in\mathbb{R}$, and $\omega(x)\ll_j(1+x)^{-j}$ for any fixed $j>0$ and $x>1$. In addition,
its Mellin transform $\wt{\omega}(u)$ has a simple pole at $u=0$ with residue $1$, and satisfies
\[
\wt{\omega}\left(\pm\frac{\alpha-\beta}2\r)=\wt{\omega}\left(\pm\frac{\gamma-\delta}2\r)=0.
\]
\end{lemma}

\subsection{Generalized Esterman $D$-function and Voronoi summation formula}\label{sectionD}
Let $V(x)$ be a bounded function.
For $q, l\in\mathbb{N}$, $h\in\mathbb{Z}$, and $\lambda\in\mathbb{C}$, define the generalized Esterman $D$-function as
\begin{equation}
\label{Ddef}
D_q\left(s,\lambda,\frac h{l}; V\r)=\sum_{(n,q)=1}\frac{\sigma_{\lambda}(n)}{n^s}e\left(n\frac{h}{l}\r)V(n).
\end{equation}
  We also adopt the convention by writing $D_q\left(s,\lambda,\frac h{l}\r)$ for $D_q\left(s,\lambda,\frac h{l}; V\r)$ when $V(x)\equiv1$. This function was originally investigated in \cite{Wu19} and further extended in \cite{Wu23}.

\begin{lemma}\label{lemgED2}
Let $q, h, l$ be integers, satisfying $(l,hq)=1$. For any fixed $\lambda\in\mathbb{C}$, $D_q\left(s,\lambda,\frac h{l}\r)$ is meromorphic as a function of $s$, satisfying the functional equation
\begin{align}\label{eqDFE}
D_q&\left(s,\lambda,\frac h{l}\r)=2(2\pi)^{-2-\lambda+2s}\Gamma\left(1-s\r) \Gamma\left(1+\lambda-s\r)\frac{\vp(q)}{q}
\sum_{a\mid q}\frac{\mu^2(a)}{\vp(a)}(la)^{1+\lambda-2s}
\\
&\ \ \ \ \ \ \ \ \ \ \times\left[\cos\left(\frac{\pi\lambda}2\r) D_{a}\B(1-s,-\lambda,\frac{\ol{ha^2}}{l}\B)-\cos\left(\pi\left(s-\frac\lambda2\r)\r) D_{a}\B(1-s,-\lambda,-\frac{\ol{ha^2}}{l}\B)\r].\notag
\end{align}
If $\lambda\neq 0$, then $D_q(s,\lambda,\frac h{l})$ has simple poles at $s=1$ and $s=1+\lambda$ with the corresponding residues being, respectively,
\[
\frac{\vp(q)}{q} l^{-1+\lambda}\zeta_q(1-\lambda),\ \ \ \ \ \ \ \ \ \ \ \ \ \ \ \ \  \frac{\vp(q)}{q}l^{-1-\lambda}\zeta_q(1+\lambda).\notag
\]
\end{lemma}
\begin{proof}
This arises as the case $r=0$ of \cite[Proposition 5.3]{Wu23}, with residues being immediate. In the functional equation given there, the $b$-sums all equal $1$ for our case. Upon expanding the Esterman $D$-function as a Dirichlet series, we factor the exponential
\[
e\B(n\frac{\ol{h_i}}{la_1}\B)=e\B(n\frac{\ol{ha_1^2}}{l}\B)e\B(n\frac{\ol{il^2}}{a_1}\B).
\]
The $i$-sum evaluates to $\mu(a_1)$ via Ramanujan summation. For fixed $a_1$, the $a$-sum in \cite[Proposition 5.3]{Wu23} becomes
\[
\sum_{a_1\mid a\mid q}\frac{\mu(a)}{a}=\frac{\vp(q)}{q}\frac{\mu(a_1)}{\vp(a_1)}.
\]
Substituting the above into the original functional equation given in \cite[Proposition 5.3]{Wu23} directly yields \eqref{eqDFE}. This completes the proof.
\end{proof}

\begin{lemma}[Voronoi formula]\label{lemEV}
Let $q, h, l$ be integers satisfying $(l,hq)=1$.
For a smooth, compactly supported function $\mV: (0,\infty)\rightarrow \mathbb{C}$ and any complex number $\lambda\neq0$, we have
\begin{align}\label{eqVoronoi}
D_q\left(s,\lambda,\frac{h}{l};\mV\r)=&\frac{\vp(q)}{q} l^{-1+\lambda}\zeta_q(1-\lambda)\wt{\mV}(1-s)+ \frac{\vp(q)}{q}l^{-1-\lambda}\zeta_q(1+\lambda)\wt{\mV}(1-s+\lambda)\\
&+\frac{\vp(q)}{q}\sum_{a\mid q}\frac{\mu^2(a)}{\vp(a)}(la)^{-1+\lambda}
\sum_{(m,a)=1}\sigma_{-\lambda}(m)e\B(\pm m\frac{\ol{ha^2}}{l}\B) \mathring{\mV}_{\pm}\left(\frac{m}{(la)^2}\r),\notag
\end{align}
where
\[
\mathring{\mV}_{\pm}(x)=2(2\pi)^{-2-\lambda}\frac{1}{2\pi i}\int_{(-2)}(4\pi^2 x)^{u-1}\wt{\mV}(u-s)\Gamma(1-u)\Gamma(1+\lambda-u)S_\pm\d u,
\]
with
\[
S_+=\cos\left(\frac{\pi \lambda}{2}\r),\quad S_{-}=\cos\left(\pi\left(u-\frac{\lambda}{2}\r)\r).
\]
If $\mV(x)=W(x/M)$, then
\[
x^j\mathring{\mV}^{(j)}_{\pm}(x)\ll_{A,j}q^{j\ve}(1+Mx)^{-A}M^{\re(s)}.
\]
\begin{proof}
Applying the Mellin transform of $\mV$, we obtain
\begin{align*}
D_q\left(s,\lambda,\frac{h}{l};\mV\r)&=\frac1{2\pi i}\int_{(2)}\wt{\mV}(u)D_q\left(s+u,\lambda,\frac{h}{l}\r)\d u\\
&=\frac1{2\pi i}\int_{(2)}\wt{\mV}(u-s)D_q\left(u,\lambda,\frac{h}{l}\r)\d u.
\end{align*}
Shifting the contour to $c_u=-2$, the residues at $u=1$ and $u=1+\lambda$ (given by Lemma \ref{lemgED2}) yield
\[
\frac{\vp(q)}{q} l^{-1+\lambda}\zeta_q(1-\lambda)\wt{\mV}(1-s), \quad \frac{\vp(q)}{q}l^{-1-\lambda}\zeta_q(1+\lambda)\wt{\mV}(1-s+\lambda),
\]
respectively. The remaining integral is handled by applying the functional equation in Lemma \ref{lemgED2}, and then expanding the resulting $D_q$ using the Dirichlet series expression given in \eqref{Ddef}, which leads to \eqref{eqVoronoi} .

 When $\mV(x)=W(x/M)$, we have
\[
\wt{\mV}(u)=M^u\wt{W}(u),
\]
and then
\begin{equation}
\label{Vint}
\mathring{\mV}_{\pm}(x)=2(2\pi)^{-2-\lambda}M^{s}\frac{1}{2\pi i}\int_{(-2)}(4\pi^2 Mx)^{u-1}\wt{W}(u-s)\Gamma(1-u)\Gamma(1+\lambda-u)S_\pm\d u.
\end{equation}
The bound
\[
\mathring{\mV}_{\pm}(x)\ll_{A}q^{\ve}(1+Mx)^{-A}M^{\re(s)}
\]
follows from shifting the line of integration in \eqref{Vint} to $c_u=-A$ ($Mx\gg q^\ve$) and $c_u=1-\ve$ ($Mx\ll q^\ve$). The bounds on the derivatives follow by differentiating $j$ times under the integral sign. This completes the proof.
\end{proof}
\end{lemma}


\section{Sketch of the treatment for moments}
Upon applying the approximate functional equation (Lemma \ref{lemafe}) and then evaluating the resulting character sums via orthogonality (Lemma \ref{lemof}), we see that $\mmM_{h,k}$ is decomposed as
\begin{equation}
\label{eqMMM}
\mmM_{h,k}(\alpha,\beta,\gamma,\delta)=\mM_{h,k}(\alpha,\beta,\gamma,\delta) +\mM_{\ol{h,k}}(\abgd),
\end{equation}
where
\begin{align*}
\mM_{h,k}(\alpha,\beta,\gamma,\delta)=\frac1{\vp^*(q)}&\sum_{d\mid q}\vp(d)\mu\left(\frac qd\r)\\
& \times \sum_{\substack{ hm\equiv \pm kn \ppmod d\\ (mn,q)=1}}\frac{1} {m_1^{\frac12+\alpha}m_2^{\frac12+\beta}n_1^{\frac12+\gamma}n_2^{\frac12+\delta}}
V_{\alpha,\beta,\gamma,\delta}\left(\frac{mn}{q^2}\r),
\end{align*}
where we set $m=m_1m_2$, $n=n_1n_2$, and
\begin{align}\label{eqMM}
\mM_{\ol{h,k}}(\abgd)=X_{\abgd}\mM_{k,h}(-\alpha,-\beta,-\gamma,-\delta).
\end{align}
We further split $\mM_{h,k}$ as
\begin{equation}
\label{Mdecomp}
\mM_{h,k}=\mM_{h,k}^D+\mM_{h,k}^+ +\mM_{h,k}^-,
\end{equation}
where $\mM_{h,k}^D$ corresponds to the contribution from the diagonal terms ($hm=kn$), while  $\mM_{h,k}^+$ and $\mM_{h,k}^-$ account for the contribution from off-diagonal terms with  $hm\equiv kn \pmod d$ and $hm\equiv -kn \pmod d$ respectively. The term $\mM_{h,k}^D$ contributes to a main term and will be treated in Section \ref{secDia}.

 Similarly, we have
\[
\mM_{\ol{h,k}}=\mM_{\ol{h,k}}^D+\mM_{\ol{h,k}}^+ +\mM_{\ol{h,k}}^-.
\]
In what follows, we shall primarily focus on the evaluation of $\mM_{h,k}$ as $\mM_{\ol{h,k}}$ can be evaluated via the relation \eqref{eqMM}.

Upon applying  M\"{o}bius inversion to eliminate the conditions $(m_i,q_d)=1$ and $(n_i,q_d)=1$, we recast the two non-diagonal terms as
\begin{equation}\label{eqMA}
\mM_{h,k}^\pm(\abgd)=\frac1{\vp^*(q)}\sum_{d\mid q}\vp(d)\mu\left(\frac qd\r)\sum_{\substack{abc\mid q^2_d\\ (a,b)=1}} \varpi_{\alpha,\beta}(ac,q_d)\varpi_{\gamma,\delta}(bc,q_d)A^\pm_{ah,bk}\left(d,\mX\r),
\end{equation}
where
\begin{equation}\label{eqdefvarpi}
\varpi_{\alpha,\beta}(a,q)=\mathop{\sum\sum}_{\substack{a_1\mid q,\ a_2\mid q\\ a_1a_2=a}}\frac{\mu(a_1)\mu(a_2)}{a_1^{\frac12+\alpha}a_2^{\frac12+\beta}}\ll a^{-\frac12+\ve},\quad \mX=\frac{q^2}{abc^2}\ge d^2,
\end{equation}
and
\[
A^\pm_{h,k}(d,\mX)=\sum_{\substack{ hm\equiv \pm kn \ppmod d\\ (mn,d)=1}}\frac{1} {m_1^{\frac12+\alpha}m_2^{\frac12+\beta}n_1^{\frac12+\gamma}n_2^{\frac12+\delta}}
V_{\alpha,\beta,\gamma,\delta}\left(\frac{mn}{\mX}\r).
\]

We partition the sum $A^+_{h,k}(d,\mX)$ into two cases depending on whether $hm<kn$ or $hm>kn$. When $hm<kn$, we apply Lemma \ref{lempartion1} to rewrite it as
\[
\sum_{\substack{ hm\equiv kn_1n_2 \ppmod d\\ (mn_1n_2,d)=1}}\frac{\sigma_{\alpha,\beta}(m)} {m^{\frac12}n_1^{\frac12+\gamma}n_2^{\frac12+\delta}}\left(\omega\left(\frac{n_1}{n_2}\r)+\omega\left(\frac{n_2}{n_1}\r)\r)
V_{\alpha,\beta,\gamma,\delta}\left(\frac{mn_1n_2}{\mX}\r).\notag
\]
When $hm>kn$, we use instead the identity
\[
\omega\left(\frac{m_1}{m_2}\r)+\omega\left(\frac{m_2}{m_1}\r)=1.
\]
 This decomposition splits $A^+_{h,k}(d,\mX)$ into four terms, obtained by permuting the shifts $h,k,\alpha,\beta,\cdots$, so that
\[
A^+_{h,k}(d,\mX)=\mA^+_{h,k}(\abgd)+\mA^+_{h,k}(\abdg)+\mA^+_{k,h}(\gdab)+\mA^+_{k,h}(\dgba),
\]
where
\[
\mA^+_{h,k}(\abgd)=\sum_{\substack{ hm\equiv kn_1n_2 \ppmod d\\ (mn_1n_2,d)=1}}\frac{\sigma_{\alpha,\beta}(m)} {m^{\frac12}n_1^{\frac12+\gamma}n_2^{\frac12+\delta}}\omega\left(\frac{n_1}{n_2}\r)
V_{\alpha,\beta,\gamma,\delta}\left(\frac{mn_1n_2}{\mX}\r).\notag
\]

We apply a similar decomposition to $A^-_{h,k}(d,\mX)$, but now using the identity
\[
\omega\left(\frac{hm}{kn}\r)+\omega\left(\frac{kn}{hm}\r)=1
\]
instead of the previous $hm<kn$ and $hm>kn$ splitting.
This partitions $A^-_{h,k}(d,\mX)$ into four terms
\[
A^-_{h,k}(d,\mX)=\mA^-_{h,k}(\abgd)+\mA^-_{h,k}(\abdg)+\mA^-_{k,h}(\gdab)+\mA^-_{k,h}(\dgba),
\]
where
\[
\mA^-_{h,k}(\abgd)=\sum_{\substack{ hm\equiv -kn_1n_2 \ppmod d\\ (mn_1n_2,d)=1}}\frac{\sigma_{\alpha,\beta}(m)} {m^{\frac12}n_1^{\frac12+\gamma}n_2^{\frac12+\delta}}\omega\left(\frac{hm}{kn_1n_2}\r)
V_{\alpha,\beta,\gamma,\delta}\left(\frac{mn_1n_2}{\mX}\r).\notag
\]

Applying the smooth dyadic partition to localize the variables $m$ and $n$, we obtain
\[
\mA^{\pm}_{h,k}(\abgd)=\sum_{M,N}B^\pm_{h,k}(\abgd;M,N),
\]
where the summands, $B^\pm_{h,k}$, inherit the structure of $\mA^{\pm}_{h,k}$,  but incorporate smooth test functions $W\left(\frac{m}{M}\r)$ and $W\left(\frac{n}{N}\r)$.
The analysis of $B^\pm_{h,k}(M,N)$ bifurcates based on the relative sizes of $hM$ and $kN$. The main terms arise from the balanced terms, in which $hM$ and $kN$ are close in magnitude. Conversely, when $hM$ and $kN$ differ significantly, these $B^\pm_{h,k}(M,N)$ (dubbed unbalanced terms) contribute only to the error terms.

For balanced terms, we demonstrate that
\[
B^\pm_{h,k}(\abgd;M,N)=(\text{Main term})_{M,N}+E^{\pm}_{h,k}(M,N),
\]
where the main terms are too intricate to explicitly state here. Instead, we focus on the summation over all $M$, $N$.
The following single upper bound (see Lemma \ref{lemupPhk})
\[
(\text{Main term})_{M,N}\ll_j q^{j\ve}(hk)^{-\frac12}\left(\frac{hM}{kN}\r)^{\frac12}\left(1+\frac{kN}{hM}\r)^{-j},\quad \text{for any}\ j>0,
\]
extends the summation to unbalanced terms with an acceptable error.
These analyses are detailed in Section \ref{secBhk}.

We estimate the error term $E^{\pm}_{h,k}(M,N)$ in Section \ref{secE}, leveraging the upper bound for sums of Kloosterman sums established in Theorem \ref{thKls1}. Let $\mE^\pm_{h,k}(M,N)$ denote the total contribution of $E^\pm_{h,k}(M,N)$ to $\mmM_{h,k}(\abgd)$. The following estimate holds.
\begin{theorem}\label{thmEhk}
For $hM\ll kN$, we have
\begin{equation}\label{eqEhk}
\mE^\pm_{h,k}(M,N)
\ll q^\ve \left(\frac{hk}{q}\r)^{\frac14}\left(\frac{kN}{hM}\r)^{\frac14}+q^\ve \left(\frac{hk}{q}\r)^{\frac12}\left(\frac{kN}{hM}\r)^{\frac12}.
\end{equation}
\end{theorem}
Let $\mE^\pm(H,K,M,N)$ denote the total contribution of $E^\pm_{h,k}(M,N)$ to $\mmM(\alpha,\beta,\gamma,\delta)$, where $h$ and $k$ range over $[H, 2H]\times[K, 2K]$. Explicitly,
\[
\mE^\pm(H,K,M,N)=\sum_{H\le h\le 2H}\sum_{K\le k\le 2K}\frac{\alpha_h\ol{\alpha_k}}{h^{\frac12}k^{\frac12}}\mE^\pm_{h,k}(M,N).
\]
A trivial estimation of the summation using \eqref{eqEhk} yields the following bound.
\begin{theorem}\label{thmEHK}
For $HM\ll KN$, we have
\begin{equation}\label{eqEHK}
\mE^\pm(H,K,M,N)\ll \frac{(HK)^{\frac34}}{q^{\frac14-\ve}}\left(\frac{KN}{HM}\r)^{\frac14}+\frac{HK}{q^{\frac12-\ve}}\left(\frac{KN}{HM}\r)^{\frac12}.
\end{equation}
\end{theorem}

Let $\mB^\pm_{h,k}(M,N)$ denote the total contribution of $B^\pm_{h,k}(M,N)$ to $\mmM_{h,k}(\abgd)$, and let $\mB^\pm(H,K,M,N)$ denote the total contribution of $B^\pm_{h,k}(M,N)$ to $\mmM(\abgd)$, where $h$ and $k$ range over $[H, 2H]\times[K, 2K]$.
When $M$ is significantly smaller compared to $N$, we can bound them via the following trivial bounds
\begin{equation}\label{eqBHK}
\mB^\pm_{h,k}(M,N)\ll q^{-1+\ve}(MN)^{\frac12}+q^\ve(M/N)^{\frac12},\quad \mB^\pm(H,K,M,N)\ll q^{-1+\ve}(HKMN)^{\frac12}.
\end{equation}

The proofs of Theorems \ref{thmmain1}--\ref{thmmain} is completed by analyzing unbalanced terms in Sections \ref{secub}--\ref{secproof12}, where $\mB^\pm_{h,k}(M,N)$ and $\mB^\pm(H,K,M,N)$ are reformulated as bilinear forms of incomplete Kloosterman sums. These are then analyzed using new bounds from \cite{KSWX23} alongside Weil's well-known bound.
\section{Some arithmetic sums}
\begin{lemma}\label{lemYa}
For complex numbers $\alpha, \beta, \gamma, \delta$ and any integer $a$, let $Y_a(\abgd)$ be defined as in \eqref{eqdefY}.
We have
\[
Y_{a}(\abgd)=Y_{a}(\bagd)=Y_{a}(\badg)=Y_{a}(\abdg).
\]
\end{lemma}
\begin{proof}
Applying the identity
\begin{align*}
1-\frac1{p^{2+\alpha+\beta+\gamma+\delta}} =&\frac{1}{1-p^{\gamma-\delta}}\left(1-\frac1{p^{1+\alpha+\delta}}\r) \left(1-\frac1{p^{1+\beta+\delta}}\r)\\
&- \frac{p^{\gamma-\delta}}{1-p^{\gamma-\delta}}\left(1-\frac1{p^{1+\alpha+\gamma}}\r) \left(1-\frac1{p^{1+\beta+\gamma}}\r),
\end{align*}
we arrive at the expression
\begin{align}\label{eqYpk}
Y_{p^k}&(\abgd)\left(1-\frac1{p^{2+\alpha+\beta+\gamma+\delta}}\r)\\
=&\frac1{p^{k\gamma}}\B(\left(1-\frac1{p^{2+\alpha+\beta+\gamma+\delta}}\r) +\frac{p^{\gamma-\delta}-p^{(k+1)(\gamma-\delta)}}{1-p^{\gamma-\delta}}\left(1-\frac1{p^{1+\alpha+\gamma}}\r) \left(1-\frac1{p^{1+\beta+\gamma}}\r)\B)\notag\\
=&\frac{p^{-(k+1)\gamma}}{p^{-\gamma}-p^{-\delta}}\left(1-\frac1{p^{1+\alpha+\delta}}\r) \left(1-\frac1{p^{1+\beta+\delta}}\r) +\frac{p^{-(k+1)\delta}}{p^{-\delta}-p^{-\gamma}}\left(1-\frac1{p^{1+\alpha+\gamma}}\r) \left(1-\frac1{p^{1+\beta+\gamma}}\r).\notag
\end{align}
This expression is symmetric in $\alpha$ and $\beta$, as well as in $\gamma$ and $\delta$. Thus,
\[
Y_{p^k}(\abgd)=Y_{p^k}(\bagd)=Y_{p^k}(\badg)=Y_{p^k}(\abdg).
\]
The general case follows from the multiplicativity of $Y_a$.
\end{proof}

The calculation of the diagonal terms requires evaluating the following arithmetic sum
\[
\mathcal{F}_a(\abgd)=\sum_{d\mid a^\infty}\frac{\sigma_{\alpha,\beta}(d)\sigma_{\gamma,\delta}(ad)}{d}.
\]
Our next result evaluates $\mathcal{F}_a(\abgd)$.
\begin{lemma}\label{lemFY}
For complex numbers $\alpha, \beta, \gamma, \delta$ and any integer $a$, we have
\begin{align*}
\mathcal{F}_a(\abgd)=&Y_a(\abgd) \prod_{p\mid a}\left(1-\frac1{p^{1+\alpha+\gamma}}\r)^{-1} \left(1-\frac1{p^{1+\beta+\gamma}}\r)^{-1} \\ &\times\left(1-\frac1{p^{1+\alpha+\delta}}\r)^{-1} \left(1-\frac1{p^{1+\beta+\delta}}\r)^{-1}\left(1-\frac1{p^{2+\alpha+\beta+\gamma+\delta}}\r).
\end{align*}
\end{lemma}
\begin{proof}
By multiplicativity, it suffices to prove the case $a=p^k$, $k\ge1$. Expanding $\sigma_{\alpha,\beta}(d)$ implies that
\[
\mathcal{F}_{p^k}(\abgd)=\sum_{i=0}^\infty\frac1{p^{i(1+\alpha)}}\sum_{j=0}^\infty\frac1{p^{j(1+\beta)}}\sigma_{\gamma,\delta}(p^{k+i+j}).
\]
Applying the identity
\[
\sigma_{\gamma,\delta}(p^{k+i+j})=\frac{p^{-(k+i+j+1)\gamma}-p^{-(k+i+j+1)\delta}}{p^{-\gamma}-p^{-\delta}},
\]
we see that
\begin{align*}
\mathcal{F}_{p^k}(\abgd)=&\B(1-\frac1{p^{1+\alpha+\gamma}}\B)^{-1} \B(1-\frac1{p^{1+\beta+\gamma}}\B)^{-1}\frac{p^{-(k+1)\gamma}}{p^{-\gamma}-p^{-\delta}}\\
&-\B(1-\frac1{p^{1+\alpha+\delta}}\B)^{-1} \B(1-\frac1{p^{1+\beta+\delta}}\B)^{-1}\frac{p^{-(k+1)\delta}}{p^{-\gamma}-p^{-\delta}}.
\end{align*}
By \eqref{eqYpk}, this simplifies to
\begin{align*}
\mathcal{F}_{p^k}(\abgd)=&Y_{p^k}(\abgd)\B(1-\frac1{p^{1+\alpha+\delta}}\B)^{-1} \B(1-\frac1{p^{1+\beta+\delta}}\B)^{-1}\\
&\B(1-\frac1{p^{1+\alpha+\gamma}}\B)^{-1} \B(1-\frac1{p^{1+\beta+\gamma}}\B)^{-1} \B(1-\frac1{p^{2+\alpha+\beta+\gamma+\delta}}\B).
\end{align*}
The general case follows from multiplicativity.
\end{proof}

  Our next lemma deals with an arithmetic sum of $Y_a(\abgd)$. For notational convenience, we define
\begin{equation}\label{eqdefmY}
\mathcal{Y}_{h,k,\abgd}(s)=Y_h(-\delta-s,\beta+s,\gamma+s,-\alpha-s)Y_k(\gamma+s,-\alpha-s,-\delta-s,\beta+s).
\end{equation}

\begin{lemma}\label{lemmG}
Let $q$ be a square-free integer and
\[
\mG_q(\abgd,s)=\sum_{\substack{abc\mid q^2\\ (a,b)=1}} \frac{\varpi_{\alpha,\beta}(ac,q)\varpi_{\gamma,\delta}(bc,q)}{a^{\frac12+s} b^{\frac12+s} c^{2s}}
\mathcal{Y}_{a,b,\abgd}(s).
\]
Then we have
\begin{align*}
\mG_q(\abgd,s)=&\prod_{p\mid q}\left(1-\frac1{p^{1-\alpha+\beta}}\r)\left(1-\frac1{p^{1+\beta+\gamma+2s}}\r) \\ &\times\left(1-\frac1{p^{1+\gamma-\delta}}\r)\left(1-\frac1{p^{2-\alpha+\beta+\gamma-\delta}}\r)^{-1}\left(1-\frac2p+\frac1{p^{1+\gamma+\delta+2s}}\r).
\end{align*}
\end{lemma}
\begin{proof}
By multiplicativity, it suffices to prove the case $q=p$ (prime). From \eqref{eqdefvarpi} and \eqref{eqdefmY}, we have
\begin{align*}
 &\varpi_{\alpha,\beta}(p,p)=-\frac1{p^{\frac12+\alpha}}-\frac1{p^{\frac12+\beta}},\quad \varpi_{\alpha,\beta}(p^2,p)=\frac1{p^{1+\alpha+\beta}},\\
 &\mathcal{Y}_{p^k,1,\abgd}(s)=Y_{p^k}(-\delta-s,\beta+s,\gamma+s,-\alpha-s),\\
 &\mathcal{Y}_{1,p^k,\abgd}(s)=Y_{p^k}(\gamma+s,-\alpha-s,-\delta-s,\beta+s).
\end{align*}
Additionally,
\[
Y_p(\abgd)\left(1-\frac1{p^{2+\alpha+\beta+\gamma+\delta}}\r)=
\frac1{p^{\gamma}}+\frac1{p^{\delta}}-\frac1{p^{1+\alpha+\gamma+\delta}}-\frac1{p^{1+\beta+\gamma+\delta}},
\]
and
\begin{align*}
Y_{p^2}(\abgd)\left(1-\frac1{p^{2+\alpha+\beta+\gamma+\delta}}\r)=&\frac1{p^{2\gamma}}\left(1-\frac1{p^{2+\alpha+\beta+\gamma+\delta}}\r)\\
&+\left(\frac1{p^{\gamma+\delta}}+\frac1{p^{2\delta}}\r)\left(1-\frac1{p^{1+\alpha+\gamma}}\r) \left(1-\frac1{p^{1+\beta+\gamma}}\r).
\end{align*}
Expanding $\mG_p$ renders that
\begin{align*}
\mG_p(&\abgd,s)=1+\frac{\varpi_{\alpha,\beta}(p,p)}{p^{\frac12+s}}\mathcal{Y}_{p,1,\abgd}(s)
+\frac{\varpi_{\gamma,\delta}(p,p)}{p^{\frac12+s}}\mathcal{Y}_{1,p,\abgd}(s)\\
&+\frac{\varpi_{\alpha,\beta}(p^2,p)}{p^{1+2s}}\mathcal{Y}_{p^2,1,\abgd}(s)+\frac{\varpi_{\gamma,\delta}(p^2,p)}{p^{1+2s}}\mathcal{Y}_{1,p^2,\abgd}(s)\\
&+\frac{\varpi_{\alpha,\beta}(p^2,p)\varpi_{\gamma,\delta}(p,p)}{p^{\frac12+3s}}\mathcal{Y}_{p,1,\abgd}(s)
+\frac{\varpi_{\alpha,\beta}(p,p)\varpi_{\gamma,\delta}(p^2,p)}{p^{\frac12+3s}}\mathcal{Y}_{1,p,\abgd}(s)\\
&+\frac{\varpi_{\alpha,\beta}(p,p)\varpi_{\gamma,\delta}(p,p)}{p^{2s}}+\frac{\varpi_{\alpha,\beta}(p^2,p)\varpi_{\gamma,\delta}(p^2,p)}{p^{4s}}.
\end{align*}
A straightforward calculation yields
\begin{align*}
\mG_p(\abgd,s)\left(1-\frac1{p^{2-\alpha+\beta+\gamma-\delta}}\r)=&\left(1-\frac1{p^{1-\alpha+\beta}}\r)\left(1-\frac1{p^{1+\beta+\gamma+2s}}\r) \\
&\times \left(1-\frac1{p^{1+\gamma-\delta}}\r)\left(1-\frac2p+\frac1{p^{1+\gamma+\delta+2s}}\r),
\end{align*}
which can be verified using a mathematical software, such as \emph{Mathematica}. The general case follows from multiplicativity.
\end{proof}

   Our last two lemmas evaluate certain Dirichlet series.
\begin{lemma}\label{lemasum}
Let $a,b$ be coprime integers. For $\re(s)>1$ and $\re(\lambda)>-1$, we have
\[
\sum_{r\ge1}\frac1{r^s}\sum_{(l,b)=1}\frac{c_{al}(r)}{l^{2+\lambda}}=a^{1-s}\prod_{p\mid a}\left(1-\frac{1}{p^{1-s}}\r) \frac{\zeta(s)\zeta_{b}(1+\lambda+s)}{\zeta_{ab}(2+\lambda)}.
\]
\end{lemma}
\begin{proof}
Let $r\rightarrow a_1b_1 r'$ and $l\rightarrow a_2l'$ with $a_1=(a^\infty,r)$, $a_2=(a^\infty,l)$, and $b_1=(b^\infty,r)$. This factorization ensures
\[
c_{al}(r)=c_{aa_2}(a_1)c_{l'}(r').
\]
Thus, the sum decomposes as
\begin{equation}
\label{sumdecomp}
\sum_{r\ge1}\frac1{r^s}\sum_{(l,b)=1}\frac{c_{al}(r)}{l^{2+\lambda}}=C(a)\prod_{p\mid b}\left(1-\frac1{p^{s}}\r)^{-1} \sum_{(r,ab)=1}\frac1{r^s}\sum_{(l,ab)=1}\frac{c_{l}(r)}{l^{2+\lambda}},
\end{equation}
where
\[
C(a)=\sum_{a_1\mid a^\infty}\sum_{a_2\mid a^\infty}\frac{c_{aa_2}(a_1)}{a_1^sa_2^{2+\lambda}}.
\]
By \cite[Lemma 6.2]{Wu23}, the inner sums equal to
\[
\sum_{(r,ab)=1}\frac1{r^s}\sum_{(l,ab)=1}\frac{c_{l}(r)}{l^{2+\lambda}}=\frac{\zeta_{ab}(s)\zeta_{ab}(1+\lambda+s)}{\zeta_{ab}(2+\lambda)}.
\]

To compute $C(a)$, it suffices to consider $a=p^k$ by multiplicativity.
Using the identities for Ramanujan sums
\begin{align*}
c_{p^j}(p^i)=\begin{cases}
0,& \text{for}\ i<j-1,\\
-p^{j-1} & \text{for}\ i=j-1,\\
\vp(p^{j}) & \text{for}\ i\ge j.
\end{cases}
\end{align*}
we write $a_1=p^i$, $a_2=p^j$ and partition the sum for $C(p^k)$ into two terms, based on whether $i=k+j-1$ or $i\ge k+j$. Explicitly,
\begin{align*}
C(p^k)&=\sum_{j=0}^\infty\frac{-p^{k+j-1}}{p^{j(2+\lambda)+(k+j-1)s}}+\sum_{j=0}^\infty\frac{\vp(p^{k+j})}{p^{j(2+\lambda)}}\sum_{i=k+j}^\infty\frac{1}{p^{is}}\\
&=\frac{-p^s}{p^{ks-k+1}}\left(1-\frac1{p^{1+\lambda+s}}\r)^{-1}+\frac{p-1}{p^{ks-k+1}}\left(1-\frac1{p^{s}}\r)^{-1}\left(1-\frac1{p^{1+\lambda+s}}\r)^{-1}\\
&=\frac{1-p^{s-1}}{p^{k(s-1)}}\left(1-\frac1{p^{s}}\r)^{-1}\left(1-\frac1{p^{1+\lambda+s}}\r)^{-1}.
\end{align*}
  It then follows from multiplicativity that for any $a$,
\[
C(a)=a^{1-s}\prod_{p\mid a}\left(1-\frac{1}{p^{1-s}}\r)\left(1-\frac1{p^{s}}\r)^{-1}\left(1-\frac1{p^{1+\lambda+s}}\r)^{-1}.
\]
Substituting this back into \eqref{sumdecomp} yields the lemma.
\end{proof}

\begin{lemma}\label{lemarithmeticq}
For any $z\in\mathbb{C}$, we have
\begin{equation*}
\sum_{d\mid q}\frac{\vp(d)^2}{d^{1+z}}\mu\left(\frac{q}{d}\r)\prod_{p\mid q_d}\left(1-\frac2p+\frac1{p^{1+z}}\r)
=\frac{\vp^*(q)}{q^{z}}\prod_{p\mid q}\left(1-\frac1{p^{1-z}}\r).
\end{equation*}
\end{lemma}
\begin{proof}
By multiplicativity, it suffices to verify the identity for $q=p$ (prime) and $q=p^k$ ($k>1$).

For $q=p$, the left-hand side (LHS) simplifies to
\[
\mu(p)\left(1-\frac2{p}+\frac1{p^{1+z}}\r)+\frac{\vp(p)^2}{p^{1+z}}=\frac{p-2}{p^{z}}\left(1-\frac1{p^{1-z}}\r),
\]
which is the same as the right-hand side (RHS) of the desired identity in this case.

For $q=p^k$ with $k>1$, the LHS reduces to
\[
\sum_{d\mid p^k}\frac{\vp(d)^2}{d^{1+z}}\mu\left(\frac{p^k}{d}\r).
\]
 Note that here the only contributions come from $d=p^{k-1}$ and $d=p^k$, so that
\[
\mu(p)\frac{\vp(p^{k-1})^2}{p^{(k-1)(1+z)}}+\frac{\vp(p^k)^2}{p^{k(1+z)}}= \frac{\vp(p^k)^2}{p^{k(1+z)}}\left(1-\frac{1}{p^{1-z}}\r).
\]
 One checks directly that right-hand side matches this expression. This therefore completes the proof.
\end{proof}

\section{Diagonal terms}\label{secDia}

  Note that when considering the contribution from the diagonal terms $hm=kn$, the summation over $d$ is simplified by the factor $\vp^*(q)$. Let $h_1=(h^\infty,m)$ and $k_1=(k^\infty,n)$.
Given that $(h,k)=1$, we have
\begin{align*}
\mM_{h,k}^D=&(hk)^{-\frac12}\sum_{\substack{h_1\mid h^\infty\\ k_1\mid k^\infty}}\frac{\sigma_{\alpha,\beta}(h_1kk_1)\sigma_{\gamma,\delta}(hh_1k_1)}{h_1k_1}
\sum_{(n,hkq)=1}\frac{\sigma_{\alpha,\beta}(n)\sigma_{\gamma,\delta}(n)}{n} V\B(\frac{hkh_1^2k_1^2n^2}{q^2}\B)\\
=&\frac{1}{2\pi i}\int_{(1)}\frac{G(s)}{s}\frac{q^{2s}g(s)}{(hk)^{\frac12+s}}\sum_{\substack{h_1\mid h^\infty\\ k_1\mid k^\infty}}\frac{\sigma_{\alpha,\beta}(h_1kk_1)\sigma_{\gamma,\delta}(hh_1k_1)}{h_1^{1+2s}k_1^{1+2s}} \sum_{(n,hkq)=1}\frac{\sigma_{\alpha,\beta}(n)\sigma_{\gamma,\delta}(n)}{n^{1+2s}} \d s.
\end{align*}
By the Ramanujan identity, the summation over $n$ evaluates to
\[
\frac{\zeta_{hkq}(1+\alpha+\gamma+2s)\zeta_{hkq}(1+\alpha+\delta+2s)\zeta_{hkq}(1+\beta+\gamma+2s)\zeta_{hkq}(1+\beta+\delta+2s)} {\zeta_{hkq}(2+\alpha+\beta+\gamma+\delta+4s)}.
\]
We shift the contour of integration to $\re(s)=-\frac14+\ve$, encountering only a simple pole at $s=0$ since all potential poles from the zeta functions are canceled by zeros of $G(s)$. The remaining integral contributes an error term of $O\left(q^{-\frac12+\ve}\r)$.

For the residue at $s=0$, the sums involving $h_1$ and $k_1$ correspond to $\mathcal{F}_h(\abgd)$ and $\mathcal{F}_k(\gdab)$ respectively. Applying Lemma \ref{lemFY} and the definition of $Z_{h,k,q}$ given in \eqref{Zdef}, we conclude that
\begin{align}
\label{MD}
\begin{split}
\mM_{h,k}^D(\abgd)=& (hk)^{-\frac12}Y_h(\abgd)Y_{k}(\gdab)Z_q(\abgd)+O\left(q^{-\frac12+\ve}\r) \\
=& (hk)^{-\frac12}Z_{h,k,q}(\abgd)+O\left(q^{-\frac12+\ve}\r).
\end{split}
\end{align}

  We evaluate $\mM_{\ol{h,k}}^D(\abgd)$ via \eqref{eqMM} to see that
\begin{equation}
\label{MDdual}
\mM_{\ol{h,k}}^D(\abgd)= (hk)^{-\frac12}X_{\abgd}Z_{h,k,q}(-\gamma,-\delta,-\alpha,-\beta)+O\left(q^{-\frac12+\ve}\r).
\end{equation}

\section{Evaluation of $B^\pm_{h,k}$}\label{secBhk}
Note that the condition $(n_2,d)=1$ in $B^\pm_{h,k}$ is automatically satisfied by other variables. We factorize $h=h_1h_2$ where $h_2=(h,n_1)$. By canceling $h_2$ in the congruence condition and performing the variable substitution $n_1\rightarrow h_2n_1$, we transform the congruence into
\begin{equation}\label{eqcongruence}
kn_1n_2=dr\pm h_1m
\end{equation}
for integers $r\ge1$.
We define the following functions
\begin{equation}\label{eqdefF+}
F^+(m)=\frac1{(dr+ h_1m)^{\frac12+\delta}}\omega\left(\frac{h_2kn^2_1}{dr+ h_1m}\r)V\left(\frac{h_2m(dr+ h_1m)}{k\mX}\r),
\end{equation}
and
\begin{equation}\label{eqdefF-}
F^-(m)=\frac1{(dr- h_1m)^{\frac12+\delta}}\omega\left(\frac{h_2kn^2_1}{dr- h_1m}\r)\omega\left(\frac{h_1m}{dr- h_1m}\r)V\left(\frac{h_2m(dr- h_1m)}{k\mX}\r).
\end{equation}
 We now eliminate $n_2$ using \eqref{eqcongruence} to see that
\[
B^\pm_{h,k}=\sum_{h_1h_2=h}\frac{k^{\frac12+\delta}}{h_2^{\frac12+\gamma}}
\sum_{r\ge1}\sum_{(n_1,dh_1)=1}n_1^{\delta-\gamma}\sum_{\substack{h_1m\equiv \mp dr\ppmod{kn_1}\\ (m,d)=1}}
\frac{\sigma_{\alpha-\beta}(m)}{m^{\frac12+\alpha}}F^\pm_{M,N}(m),
\]
where
\begin{equation}\label{eqdefFpm}
F^\pm_{M,N}(m)=F^\pm(m) W\left(\frac{m}{M}\r)W\left(\frac{h_2(dr\pm h_1m)}{kN}\r).
\end{equation}

For any divisor $g\mid kn_1$,  we factorize $g=k_1l$ where $l=(g,n_1)$ and $k_1\mid k$. This establishes a bijection between $g$ and the ordered pairs $(k_1,l)$.
Leveraging this decomposition, we eliminate the congruence condition using primitive additive characters
\begin{align*}
\bm{1}_{h_1m\equiv \mp dr \ppmod {kn_1}}&=\frac1{kn_1}\sum_{g\mid kn_1}\ssum_{t \ppmod {g}}e\left(\frac{drt\pm h_1mt}{g}\right)\\
&=\frac1{kn_1}\sum_{k_1\mid k}\sum_{\substack{fl= n_1\\ (f,k_1)=1}}\mathop{\sum\nolimits^*}_{t\ppmod {k_1l}}e\left(\frac{drt\pm h_1mt}{k_1l}\right).
\end{align*}
Substituting this into the expression for $B^\pm_{h,k}$, we obtain that
\begin{align}\label{eqBlast}
B^\pm_{h,k}&=\frac{1}{k^{\frac12-\delta}}\sum_{h_1h_2= h}\sum_{k_1\mid k} \frac1{h_2^{\frac12+\gamma}}  \\
&\times\sum_{r\ge1}\sum_{(n_1,dh_1)=1}\frac1{n_1^{1+\gamma-\delta}}\sum_{\substack{fl= n_1\\ (f,k_1)=1}}\ssum_{t\ppmod {k_1l}}e\left(\frac{drt}{k_1l}\right)D_d\left(\tfrac12+\alpha,\alpha-\beta,\frac{\pm h_1t}{k_1l};F^\pm_{M,N}\r).\notag
\end{align}

By Lemma \ref{lemEV}, the dominant terms of $D_d$ are given by
\[
\frac{\vp(d)}{d} \frac{\zeta_d(1-\alpha+\beta)}{(k_1l)^{1-\alpha+\beta}}\wt{F}^\pm_{M,N}\left(\tfrac12-\alpha\r),\quad
\frac{\vp(d)}{d} \frac{\zeta_d(1+\alpha-\beta)}{(k_1l)^{1+\alpha-\beta}}\wt{F}^\pm_{M,N}\left(\tfrac12-\beta\r).
\]
Substituting these into \eqref{eqBlast}, we denote by $P^\pm_{h,k}(\abgd;M,N)$ the contribution of the first term above, while noting that the sum over $t$ there evaluates to $c_{k_1l}(r)$. By symmetry, we obtain
\[
B^\pm_{h,k}(\abgd;h,k)=P^\pm_{h,k}(\abgd;M,N)+P^\pm_{h,k}(\bagd;M,N)+E^\pm_{h,k}(M,N),
\]
where
\begin{align}\label{eqdefPhk}
P^\pm_{h,k}(\abgd;M,N)
=&\frac{\vp(d)}{d}\frac{\zeta_q(1-\alpha+\beta)}{k^{\frac12-\delta}}\sum_{h_1h_2= h}\sum_{k_1\mid k} \frac1{h_2^{\frac12+\gamma}}\frac1{k_1^{1-\alpha+\beta}} \\
&\times
\sum_{r\ge1}\sum_{(f,dh_1k_1)=1}\frac1{f^{1+\gamma-\delta}}
\sum_{(l,dh_1)=1}\frac{c_{k_1l}(r)}{l^{2-\alpha+\beta+\gamma-\delta}}\wt{F}^\pm_{M,N}\left(\tfrac12-\alpha\r).\notag
\end{align}
We focus on analyzing $P^\pm_{h,k}$  in the remainder of this section and defer the treatment of the error term $E^\pm_{h,k}$ to Section \ref{secE}.

\subsection{Upper bound for main terms}
For the terms $P^\pm_{h,k}(\abgd;M,N)$, we establish the following upper bound that permits extension of the main terms to all $M$ and $N$.
\begin{lemma}\label{lemupPhk}
For any $j>0$, we have
\[
P^\pm_{h,k}(\abgd;M,N)\ll_j q^{j\ve}\frac{(MN)^{\frac12}}{d}\left(1+\frac{kN}{hM}\r)^{-j}.
\]
\end{lemma}

\begin{proof}
Observe from \eqref{eqdefFpm} that for any $j>0$,
\begin{equation}\label{equppF}
F^\pm_{M,N}(m)\ll_j q^{j\ve}\left(\frac{h_2}{kN}\r)^{\frac12}\left(1+\frac{kN}{hM}\r)^{-j},
\end{equation}
 where the dominant factor follows directly from \eqref{eqdefF+} and \eqref{eqdefF-}. For $F^+_{M,N}(m)$, the rapid decay in \eqref{equppF} originates from the $W$-function in \eqref{eqdefFpm} with $r\ge1$, while for $F^-_{M,N}(m)$ it arises from the second $\omega$-function in \eqref{eqdefF-}.
This rapid decay property also implies that the number of terms $r$ is $O(kN/h_2d)$.

 We denote the Mellin transform $\wt{F}^\pm_{M,N}$ of $F^\pm_{M,N}$ by
\[
\wt{F}^\pm_{M,N}(u)=\int_0^\infty F^\pm_{M,N}(x)x^{u-1}\d x.
\]
Applying \eqref{equppF} to the above yields
\[
\wt{F}^\pm_{M,N}\left(\tfrac12-\alpha\r),\wt{F}^\pm_{M,N}\left(\tfrac12-\beta\r)\ll_j q^{j\ve}\left(\frac{h_2M}{kN}\r)^{\frac12}\left(1+\frac{kN}{hM}\r)^{-j}.
\]
Substituting this into \eqref{eqdefPhk} and performing trivial summation
with the number of $r$ not exceeding $O(kN/h_2d)$ allow us to complete the proof.
\end{proof}

We now apply Theorem \ref{thmEhk} and Lemma \ref{lemupPhk} to extend the summation range of the main terms to include unbalanced terms with controlled error terms.
After summing over all $M$ and $N$ using \eqref{eqdyadic}, we proceed with the final analyses of the main terms.
We define
\[
\mP_{h,k}^\pm(\abgd)=\sum_{M,N}P^\pm_{h,k}(\abgd;M,N).
\]
It follows from \eqref{eqMA} that its contribution to $\mM^\pm_{h,k}(\abgd)$ is given by
\begin{align}\label{eqPP}
\mQ_{h,k}^\pm(\abgd)=\frac1{\vp^*(q)}&\sum_{d\mid q}\vp(d)\mu\left(\frac qd\r)\\
&\times\sum_{\substack{abc\mid q^2_d\\ (a,b)=1}} \varpi_{\alpha,\beta}(ac,q_d)\varpi_{\gamma,\delta}(bc,q_d)\mP_{ah,bk}^\pm(\abgd),\notag
\end{align}
which will be evaluated in the following sections.

\subsection{Main terms with negative sign} \label{secMneg}
We compute $\mP_{h,k}^-$ and subsequently $\mQ_{h,k}^-$ in this section. The following lemma furnishes an explicit expression for $\mP_{h,k}^-$.
\begin{lemma}\label{lemP}
With $c_s=\frac14$, $c_v=\ve$, we have
\begin{align*}
\mP_{h,k}^-&(\abgd)=\frac{\vp(d)}{d}(hk)^{-\frac12}\zeta_d(1-\alpha+\beta)\zeta_d(1+\gamma-\delta)\frac{1}{(2\pi i)^2 }\int_{(c_s)}\int_{(c_v)}
\frac{\mX^{s}} {d^{\alpha+\delta+2s}}\\
&\times H^{-}(s,v)\mathcal{Y}_{h,k,\abgd}(s) \frac{\zeta(\alpha+\delta+2s)\zeta_{d}(1+\beta+\gamma+2s)}{\zeta_{d}(2-\alpha+\beta+\gamma-\delta)}\d s\d v+O\B(\frac{k^{\frac14}\mX^{\frac12+\ve}}{dh^{\frac14}}\B),
\end{align*}
where $\mathcal{Y}_{h,k,\abgd}(s)$ is defined as in \eqref{eqdefmY}, and
\[
H^{-}(s,v)=\frac{G(s)}{s}g(s)\wt{\omega}(v)\frac{\Gamma(\frac12-\alpha-s-v)\Gamma(\frac12-\delta-s+v)}{\Gamma(1-\alpha-\delta-2s)}.
\]
\end{lemma}
\begin{proof}
We apply the partition identity \eqref{eqdyadic} to see that
\begin{align}\label{eqppm}
\mP_{h,k}^-(\alpha,\beta,\gamma,\delta)
=&\frac{\vp(d)}{d}\frac{\zeta_d(1-\alpha+\beta)}{k^{\frac12-\delta}}\sum_{h_1h_2= h}\sum_{k_1\mid k} \frac1{h_2^{\frac12+\gamma}} \frac1{k_1^{1-\alpha+\beta}}\\
&\times \sum_{r\ge1}
\sum_{(f,dh_1k_1)=1}\frac1{f^{1+\gamma-\delta}}
\sum_{(l,dh_1)=1}\frac{c_{k_1l}(r)}{l^{2-\alpha+\beta+\gamma-\delta}}\wt{F}^-\left(\tfrac12-\alpha\r).\notag
\end{align}
We analytically separate the variables $r, f$ and $l$ in $F^-$ by taking Mellin transforms of $\omega$ and $V$, obtaining
\[
F^-(x)=\frac1{(2\pi i)^3}\int_{(c_s)}\int_{(c_u)}\int_{(c_v)}\frac{\mX^{s}k^{s-u}(dr-h_1x)^{-\frac12-\delta-s+u+v}} {h_2^{u+s}h_1^vf^{2u}l^{2u}x^{s+v}}\wt{\omega}(u)\wt{\omega}(v)\frac{G(s)}{s}g(s)\d s\d u\d v.
\]

Initially setting $c_s=c_u=c_v=\ve$, we further separate variables in $(dr-h_1x)^{-\frac12-\delta-s+u+v}$ using the following formula
\begin{equation}\label{eqGamma-}
(1-x)^{-b}=\frac1{2\pi i}\int_{(c_w)}\frac{\Gamma(w)\Gamma(1-b)}{\Gamma(1-b+w)}x^{-w}\d w,
\end{equation}
which is valid for $0<x,  c_w, \re(b)<1$.

This then leads to
\begin{align*}
F^-(x)=\frac1{(2\pi i)^4}\int_{(c_s)}\int_{(c_u)}&\int_{(c_v)}\int_{(c_w)}\frac{\mX^{s}k^{s-u}(dr)^{-\frac12-\delta-s+u+v+w}} {h_2^{s+u}h_1^{v+w} f^{2u}l^{2u}x^{s+v+w}} \\
&\times\frac{G(s)}{s}g(s)\wt{\omega}(u)\wt{\omega}(v)\frac{\Gamma(w)\Gamma(\frac12-\delta-s+u+v)}{\Gamma(\frac12-\delta-s+u+v+w)}\d s \d u \d v \d w,
\end{align*}
which allows us to write its Mellin transform as
\begin{equation}\label{eqwtF}
\wt{F}^-\left(\tfrac12-\alpha\r)=\frac1{(2\pi i)^3}\int_{(c_s)}\int_{(c_u)}\int_{(c_v)}\frac{\mX^{s}k^{s-u}(dr)^{-\alpha-\delta-2s+u}} {h_2^{s+u}h_1^{\frac12-\alpha-s}f^{2u}l^{2u}}H^{-}_1(s,u,v)\d s\d u \d v,
\end{equation}
where
\[
H_1^{-}(s,u,v)=\frac{G(s)}{s}g(s)\wt{\omega}(u)\wt{\omega}(v)\frac{\Gamma(\frac12-\alpha-s-v)\Gamma(\frac12-\delta-s+u+v)}{\Gamma(1-\alpha-\delta-2s+u)}.
\]

Substituting \eqref{eqwtF} into \eqref{eqppm} yields
\begin{align*}
\mP_{h,k}^-=\frac{\vp(d)}{d}&\frac{\zeta_q(1-\alpha+\beta)}{h^\frac12 k^{\frac12}}\\
&\times\frac1{(2\pi i)^3}\int_{(c_s)}\int_{(c_u)}\int_{(c_v)}H_{1}^{-}(s,u,v)
\frac{\mX^{s}k^{\delta+s-u}} {d^{\alpha+\delta+2s-u}h^{\gamma+s+u}}
\sum_{h_1\mid h}\sum_{k_1\mid k}  \frac{h_1^{\alpha+\gamma+2s+u}}{k_1^{1-\alpha+\beta}}
\\
&\times\sum_{(f,dh_1k_1)=1}\frac1{f^{1+\gamma-\delta+2u}}
\sum_{r\ge1}\frac1{r^{\alpha+\delta+2s-u}}\sum_{(l,dh_1)=1}\frac{c_{k_1l}(r)}{l^{2-\alpha+\beta+\gamma-\delta+2u}}\d s\d u\d v.
\end{align*}
We shift the contours of the integral to $c_s=\tfrac12+\ve$ and $c_u=\ve$, crossing the only simple pole at $s=\frac12-\alpha-v$. We first evaluate the new contour integral, leaving the residue for later.

For the new contour integral, the sums over $f, l$, and $r$ converge absolutely.
The $f$-sum evaluates to
\[
\prod_{p\mid h_1k_1}\left(1-\frac1{p^{1+\gamma-\delta+2u}}\r)\zeta_{d}(1+\gamma-\delta+2u).
\]
By Lemma \ref{lemasum}, the innermost two sums over $l$ and $r$ simplify to
\[
k_1^{1-\alpha-\delta-2s+u} \prod_{p\mid k_1}\left(1-\frac1{p^{1-\alpha-\delta-2s+u}}\r) \frac{\zeta(\alpha+\delta+2s-u)\zeta_{dh_1}(1+\beta+\gamma+2s+u)}{\zeta_{dh_1k_1}(2-\alpha+\beta+\gamma-\delta+2u)}.
\]
After these simplifications, the factor involving $h$ in the integral becomes
\begin{align}\label{eqYh1}
&\frac1{h^{\gamma+s+u}}\sum_{h_1\mid h}h_1^{\alpha+\gamma+2s+u}\prod_{p\mid h_1}\left(1-\frac1{p^{1+\gamma-\delta+2u}}\r) \left(1-\frac1{p^{1+\beta+\gamma+2s+u}}\r) \left(1-\frac1{p^{2-\alpha+\beta+\gamma-\delta+2u}}\r)^{-1}\\
=&Y_h(-\delta-s+u,\beta+s,\gamma+s+u,-\alpha-s).\notag
\end{align}
Similarly, the factor involving $k$ simplifies to
\begin{align}\label{eqYk1}
&k^{\delta+s-u}\sum_{k_1\mid k}\frac1{k_1^{\beta+\delta+2s-u}}\prod_{p\mid k_1}\left(1-\frac1{p^{1+\gamma-\delta+2u}}\r) \left(1-\frac1{p^{1-\alpha-\delta-2s+u}}\r) \left(1-\frac1{p^{2-\alpha+\beta+\gamma-\delta+2u}}\r)^{-1}\\
=&Y_k(\gamma+s+u,-\alpha-s,-\delta-s+u,\beta+s).\notag
\end{align}
In conclusion, the new contour integral evaluates to
\begin{align}\label{eqCI}
&\frac{\vp(d)}{d}(hk)^{-\frac12}\zeta_d(1-\alpha+\beta)\frac{1}{(2\pi i)^3 }\int_{(c_s)}\int_{(c_u)}\int_{(c_v)}
\frac{\mX^{s}} {d^{\alpha+\delta+2s-u}}H^{-}_{1}(s,u,v)\\
&\times Y_h(-\delta-s+u,\beta+s,\gamma+s+u,-\alpha-s)Y_k(\gamma+s+u,-\alpha-s,-\delta-s+u,\beta+s)\notag\\
&\times \frac{\zeta_d(1+\gamma-\delta+2u)\zeta(\alpha+\delta+2s-u)\zeta_{d}(1+\beta+\gamma+2s+u)}{\zeta_{d}(2-\alpha+\beta+\gamma-\delta+2u)}\d s\d u\d v.\notag
\end{align}
We then shift the contour of the integral to $c_s=\frac14$, crossing the pole at $s=\frac12-\alpha-v$ again. Note here that the pole of $\zeta(\alpha+\delta+2s-u)$ is canceled out by the zero of $H_1^{-}(s,u,v)$.

Now that we have crossed the pole at $s=\frac12-\alpha-v$ twice, we observe that the two residues ultimately cancel each other out. Let us examine this cancellation in detail. Clearly, the expression simplifies to
\begin{equation*}
\displaystyle \text{Res}_{s=\tfrac 12-\alpha-v}H_1^{-}\left(s,u,v\r)=\frac{G\left(\tfrac12-\alpha-v\r)}{\tfrac12-\alpha-v}g\left(\tfrac12-\alpha-v\r)\wt{\omega}(u)\wt{\omega}(v),
\end{equation*}
 As $G\left(\frac12-\alpha\r)=0$, it follows that for any fixed $u\neq0$ the function
$\displaystyle \text{Res}_{s=\tfrac 12-\alpha-v} H_1^{-}\left(s,u,v\r)$
remains analytic in the half-plane $\re(v)<\frac12-\re(\alpha)$. For the first residue computation, we can take $c_v$ to be negative, thereby guaranteeing absolute convergence for all sums over $f, l$, and $r$. Upon evaluating these sums as previously done, it becomes evident that the two residues at $s=\frac12-\alpha-v$ annihilate each other.

We proceed to shift the contours of the remaining integral to $c_s=\frac14$ and $c_v=\frac14-\ve$, followed by $c_u=-\frac12+2\ve$. During this process, we encounter the only pole at $u=0$, keeping in mind that the pole of $\zeta(1+\gamma-\delta+2u)$ is canceled out by $\wt{\omega}\left(\frac{\delta-\gamma}{2}\r)=0$.
The main contribution comes from the residue at $u=0$. For the integral on the new lines, analysis of \eqref{eqYh1} and \eqref{eqYk1} implies that
\[
Y_h(-\delta-s+u,\beta+s,\gamma+s+u,-\alpha-s)\ll h^{\frac14+\ve},\ \ Y_k(\gamma+s+u,-\alpha-s,-\delta-s+u,\beta+s)\ll k^{\frac34+\ve}.
\]
A straightforward estimation then shows that the integral on the new lines is bounded by $O\left(\frac{k^{\frac14}\mX^{\frac12+\ve}}{dh^{\frac14}}\r)$.
Therefore, we obtain the final expression
\begin{align*}
\mP_{h,k}^-=&\frac{\vp(d)}{d}(hk)^{-\frac12}\zeta_d(1-\alpha+\beta)\zeta_d(1+\gamma-\delta)\frac{1}{(2\pi i)^2 }\int_{(c_s)}\int_{(c_v)}
\frac{\mX^{s}} {d^{\alpha+\delta+2s}}H^{-}(s,v)\\
&\times \mathcal{Y}_{h,k,\abgd}(s) \frac{\zeta(\alpha+\delta+2s)\zeta_{d}(1+\beta+\gamma+2s)}{\zeta_{d}(2-\alpha+\beta+\gamma-\delta)}\d s\d v+O\B(\frac{k^{\frac14}\mX^{\frac12+\ve}}{dh^{\frac14}}\B).
\end{align*}
 This completes the proof of the lemma.
\end{proof}

 We now apply Lemma \ref{lemP} to \eqref{eqPP} to derive the following expression for $\mQ_{h,k}^-$.
\begin{proposition}\label{proP-}
With $c_s=\frac14$, $c_v=\ve$, we have
\begin{align*}
\mQ_{h,k}^-(\alpha,&\beta,\gamma,\delta)=(hk)^{-\frac12} \frac{\zeta_q(1-\alpha+\beta)\zeta_q(1+\gamma-\delta)}{\zeta_{q}(2-\alpha+\beta+\gamma-\delta)} \frac1{2\pi i} \int_{(c_s)}\frac{G(s)}{s}g(s)\mH_{\alpha,\delta}^{-}(s)\\
&\times \mathcal{Y}_{h,k,\abgd}(s)\zeta_q(1-\alpha-\delta-2s)\zeta_{q}(1+\beta+\gamma+2s) \d s+O\left(h^{-\frac14}k^{\frac14}q^{-\frac12+\ve}\r),
\end{align*}
where
\begin{align*}
\mH^{-}_{\alpha,\delta}(s)=&\pi^{2s-\frac12}\left(\frac{q}{\pi}\r)^{-\alpha-\delta} \frac{\Gamma\left(\frac{1-\alpha-\delta-2s}{2}\r)}{\Gamma\left(\frac{\alpha+\delta+2s}{2}\r)}\\
&\times\frac1{2\pi i} \int_{(c_v)}\frac{\Gamma\left(\frac12-\alpha-s-v\r)\Gamma\left(\frac12-\delta-s+v\r)}{\Gamma(1-\alpha-\delta-2s)}\wt{\omega}(v)dv.
\end{align*}
\end{proposition}
\begin{proof}
We begin by substituting the expression of $\mP_{ah,bk}^-(\abgd)$, obtained from  Lemma \ref{lemP}, into the formula \eqref{eqPP}. The resulting sum over $a,b,c$ transforms into $\mG_{q_d}(\abgd,s)$ through the multiplicativity of $\mathcal{Y}$.
Given that $q_d$ is square-free due to the presence of $\mu(q/d)$, we apply Lemma \ref{lemmG} to $\mG_{q_d}$ to see that
\begin{align*}
\mQ_{h,k}^-=&(hk)^{-\frac12}\zeta_q(1-\alpha+\beta)\zeta_q(1+\gamma-\delta)\frac{1}{(2\pi i)^2 }\int_{(c_s)}\int_{(c_v)}
q^{2s}H^{-}(s,v)\\
&\times\frac1{\vp^*(q)}\sum_{d\mid q}\frac{\vp(d)^2}{d^{1+\alpha+\delta+2s}}\mu\left(\frac{q}{d}\r)
\prod_{p\mid q_d}\left(1-\frac2p+\frac1{p^{1+\alpha+\delta+2s}}\r)\\
&\times\mathcal{Y}_{h,k,\abgd}(s) \frac{\zeta(\alpha+\delta+2s)\zeta_{q}(1+\beta+\gamma+2s)}{\zeta_{q}(2-\alpha+\beta+\gamma-\delta)}\d s\d v+O\left(h^{-\frac14}k^{\frac14}q^{-\frac12+\ve}\r).
\end{align*}
Next, we evaluate the sum over $d$ using Lemma \ref{lemarithmeticq}, yielding
\begin{align*}
\mQ_{h,k}^-=&(hk)^{-\frac12} \frac{\zeta_q(1-\alpha+\beta)\zeta_q(1+\gamma-\delta)}{\zeta_{q}(2-\alpha+\beta+\gamma-\delta)} \frac{q^{-\alpha-\delta}}{(2\pi i)^2 }\int_{(c_s)}\int_{(c_v)}H^{-}(s,v)\mathcal{Y}_{h,k,\abgd}(s)\\
\times& \prod_{p\mid q}\left(1-\frac1{p^{1-\alpha-\delta-2s}}\r)\zeta(\alpha+\delta+2s)\zeta_{q}(1+\beta+\gamma+2s) \d s\d v
+O\left(h^{-\frac14}k^{\frac14}q^{-\frac12+\ve}\r).
\end{align*}
Finally, we apply the functional equation
\[
\prod_{p\mid q} \left(1-\frac1{p^{1-\alpha-\delta-2s}}\r) \zeta(\alpha+\delta+2s)=\frac1{\pi^{\frac12-\alpha-\delta-2s}}\frac{\Gamma\left(\frac{1-\alpha-\delta-2s}{2}\r)} {\Gamma\left(\frac{\alpha+\delta+2s}{2}\r)}\zeta_q(1-\alpha-\delta-2s)
\]
to establish the proposition.
%
\end{proof}

\subsection{Main terms with positive sign}
The proofs for $\mP_{h,k}^+$ and $\mQ_{h,k}^+$ follow from analogous arguments to the negative sign case, subject to appropriate modifications. This next lemma  constitutes a positive-sign counterpart to Lemma \ref{lemP}.
\begin{lemma}\label{lemP+}
With $c_s=\frac12+\ve$, we have
\begin{align*}
\mP_{h,k}^+(\abgd)=&\frac{\vp(d)}{d}(hk)^{-\frac12}\zeta_d(1-\alpha+\beta)\zeta_d(1+\gamma-\delta)\frac{1}{2\pi i }\int_{(c_s)}
\frac{\mX^{s}} {d^{\alpha+\delta+2s}}H^{+}(s)\\
&\times \mathcal{Y}_{h,k,\abgd}(s) \frac{\zeta(\alpha+\delta+2s)\zeta_{d}(1+\beta+\gamma+2s)}{\zeta_{d}(2-\alpha+\beta+\gamma-\delta)}\d s+O\B(\frac{k^{\frac14}\mX^{\frac12+\ve}}{dh^{\frac14}}\B),
\end{align*}
where $\mathcal{Y}_{h,k,\abgd}(s)$ is defined as in \eqref{eqdefmY}, and
\begin{equation*}
H^{+}(s)=\frac{G(s)}{s}g(s) \frac{\Gamma(\frac12-\alpha-s)\Gamma(\alpha+\delta+2s)}{\Gamma(\frac12+\delta+s)}.
\end{equation*}
\end{lemma}
\begin{proof}
The analysis closely follows that of the proof of Lemma \ref{lemP}, with several modifications made below. We use the identity
\[
(1+x)^{-b}=\frac1{2\pi i}\int_{(c_w)}\frac{\Gamma(w)\Gamma(b-w)}{\Gamma(b)}x^{-w}\d w
\]
for $0<c_w<\re(b)$ in place of \eqref{eqGamma-}. Since there is only one $\omega$-function in $F^+$, the gamma factor evaluates to
\begin{equation*}
H_1^{+}(s,u)=\frac{\Gamma(\frac12-\alpha-s)\Gamma(\alpha+\delta+2s-u)}{\Gamma(\frac12+\delta+s-u)}\frac{G(s)}{s}g(s)\wt{\omega}(u).
\end{equation*}
Here the pole of $\Gamma(\frac12-\alpha-s)$ is canceled out by $G(\frac12-\alpha)=0$, and we shift the $s$-integral to $c_s=\frac12+\ve$ to ensure absolute convergence, without crossing any pole. After performing all summations, we obtain an analogy of \eqref{eqCI}. The $u$-contour is then shifted to $c_u=-\frac12+\ve$, crossing only the pole at $u=0$. The residue yields the main term of $\mP^{+}_{h,k}$, while the remaining integrals are bound by the error term using identical arguments as the negative sign case.
\end{proof}

  We now apply Lemma \ref{lemP+} to \eqref{eqPP} to obtain the corresponding expression for $\mQ_{h,k}^+$.
\begin{proposition}\label{proP+}
With $c_s=\frac14$, we have
\begin{align*}
\mQ_{h,k}^+(\alpha,&\beta,\gamma,\delta)= (hk)^{-\frac12} \frac{\zeta_q(1-\alpha+\beta)\zeta_q(1+\gamma-\delta)}{\zeta_{q}(2-\alpha+\beta+\gamma-\delta)} \frac1{2\pi i} \int_{(c_s)}\frac{G(s)}{s}g(s)H_{\alpha,\delta}^{+}(s)\\
&\times\mathcal{Y}_{h,k,\abgd}(s)\zeta_q(1-\alpha-\delta-2s)\zeta_{q}(1+\beta+\gamma+2s) \d s+O\left((hk)^{-\frac12}q^{-\frac12+\ve}\r),
\end{align*}
where
\begin{equation*}
H^{+}_{\alpha,\delta}(s)=\pi^{2s-\frac12}\left(\frac{q}{\pi}\r)^{-\alpha-\delta}\frac{\Gamma\left(\frac{1-\alpha-\delta-2s}{2}\r)}{\Gamma\left(\frac{\alpha+\delta+2s}{2}\r)} \frac{\Gamma(\frac12-\alpha-s)\Gamma(\alpha+\delta+2s)}{\Gamma(\frac12+\delta+s)}.
\end{equation*}
\end{proposition}
\begin{proof}
This proof follows along the same line of argument as Proposition \ref{proP-}, with Lemma \ref{lemP+} substituted for Lemma \ref{lemP}. The final step involves shifting the contour of integration from $c_s=\frac12+\ve$ to $c_s=\frac14$, which introduces no additional residues since the potential poles at $s=\frac12(1-\alpha-\delta)$ and $s=\frac12-\alpha$ of $H_{\alpha,\delta}^{+}(s)$  are precisely canceled by zeros of $\zeta_q(1-\alpha-\delta-2s)$ and $G(s)$ respectively.
\end{proof}

\section{Assembling the main terms}
We establish the following combination
\begin{align}\label{eqQP4}
&\mQ_{h,k}^+(\abgd)+\mQ_{k,h}^+(\dgba)+\mQ_{h,k}^-(\abgd)+\mQ_{k,h}^-(\dgba)\\
=& Q_{h,k}(\abgd)+O\left(h^{-\frac14}k^{\frac14}q^{-\frac12+\ve}\r).\notag
\end{align}
This yields the full contribution of main terms to $\mM_{h,k}$ as
\begin{equation*}
Q_{h,k}(\abgd)+Q_{h,k}(\bagd)+Q_{h,k}(\abdg)+Q_{h,k}(\badg).
\end{equation*}
\begin{lemma}\label{lemQ}
We have
\begin{align*}
Q_{h,k}(\abgd)=&(hk)^{-\frac12} \frac{\zeta_q(1-\alpha+\beta)\zeta_q(1+\gamma-\delta)}{\zeta_{q}(2-\alpha+\beta+\gamma-\delta)} \frac1{2\pi i} \int_{(\frac14)}\frac{G(s)}{s}g(s)H_{\alpha,\delta}(s)\\
&\times \mathcal{Y}_{h,k,\abgd}(s)\zeta_q(1-\alpha-\delta-2s)\zeta_{q}(1+\beta+\gamma+2s) \d s,
\end{align*}
where
\[
H_{\alpha,\delta}(s)= \pi^{2s-\frac12}\left(\frac{q}{\pi}\r)^{-\alpha-\delta}\frac{\Gamma\left(\frac{\frac12-\alpha-s}2\r)}{\Gamma\left(\frac{\frac12+\alpha+s}2\r)} \frac{\Gamma\left(\frac{\frac12-\delta-s}2\r)}{\Gamma\left(\frac{\frac12+\delta+s}2\r)}.
\]
\end{lemma}
\begin{proof}
From \eqref{eqdefmY} and Lemma \ref{lemYa}, we immediately obtain the symmetry relation
\[
\mathcal{Y}_{h,k,\abgd}(s)=\mathcal{Y}_{k,h,\dgba}(s).
\]
Consequently, by Proposition \ref{proP-} and Proposition \ref{proP+}, all terms on the left-hand side of \eqref{eqQP4} share identical main contribution up to gamma function factors. Specifically, we have
\[
H^{-}_{\alpha,\delta}(s)+H^{-}_{\delta,\alpha}(s)=\pi^{2s-\frac12}\left(\frac{q}{\pi}\r)^{-\alpha-\delta}\frac{\Gamma\left(\frac{1-\alpha-\delta-2s}{2}\r)}{\Gamma\left(\frac{\alpha+\delta+2s}{2}\r)} \frac{\Gamma(\frac12-\alpha-s)\Gamma(\frac12-\delta-s)}{\Gamma(1-\alpha-\delta-2s)}.
\]
The complete gamma factor combination is then given by
\begin{align*}
H_{\alpha,\delta}(s)&=H^{+}_{\alpha,\delta}(s)+H^{+}_{\delta,\alpha}(s)+H^{-}_{\alpha,\delta}(s)+H^{-}_{\delta,\alpha}(s)\\
&=\pi^{2s-\frac12}\left(\frac{q}{\pi}\r)^{-\alpha-\delta}\frac{\Gamma\left(\frac{1-\alpha-\delta-2s}{2}\r)}{\Gamma\left(\frac{\alpha+\delta+2s}{2}\r)} \Gamma_{\alpha,\delta}(s),
\end{align*}
where
\begin{align*}
\Gamma_{\alpha,\delta}(s)=&\frac{\Gamma(\frac12-\alpha-s)\Gamma(\alpha+\delta+2s)}{\Gamma(\frac12+\delta+s)}\\
&+\frac{\Gamma(\frac12-\alpha-s)\Gamma(\frac12-\delta-s)}{\Gamma(1-\alpha-\delta-2s)}
+\frac{\Gamma(\frac12-\delta-s)\Gamma(\alpha+\delta+2s)}{\Gamma(\frac12+\alpha+s)}.
\end{align*}
The proof is completed by verifying the identity
\[
\Gamma_{\alpha,\delta}(s)=\pi^{\frac12}\frac{\Gamma\left(\frac{\alpha+\delta+2s}{2}\r)} {\Gamma\left(\frac{1-\alpha-\delta-2s}{2}\r)} \frac{\Gamma\left(\frac{\frac12-\alpha-s}2\r)}{\Gamma\left(\frac{\frac12+\alpha+s}2\r)} \frac{\Gamma\left(\frac{\frac12-\delta-s}2\r)}{\Gamma\left(\frac{\frac12+\delta+s}2\r)},
\]
which follows from standard gamma function trigonometric identities as established in \cite[Lemma 8.2]{You11}.
\end{proof}

The main term of $\mM_{\ol{h,k}}$ is then determined by the relation \eqref{eqMM} as follows
\[
Q_{\ol{h,k}}(\abgd)+Q_{\ol{h,k}}(\bagd)+Q_{\ol{h,k}}(\abdg)+Q_{\ol{h,k}}(\badg),
\]
where the individual components are defined by
\begin{align}\label{eqQol}
Q_{\ol{h,k}}(\abgd)=X_{\abgd}Q_{k,h}(-\alpha,-\beta,-\gamma,-\delta).
\end{align}

\begin{lemma} \label{1stQ}
We have
\[
Q_{h,k}(\abgd)+Q_{\ol{h,k}}(\badg)=(hk)^{-\frac12}Z_{h,k,q}(-\delta,\beta,\gamma,-\alpha).
\]
\end{lemma}
\begin{proof}
By Lemma \ref{lemQ} and the relationship \eqref{eqQol}, we obtain the explicit expression
\begin{align*}
Q_{\ol{h,k}}(\badg)&=\frac{X_{\abgd}}{(hk)^{\frac12}} \frac{\zeta_q(1-\alpha+\beta)\zeta_q(1+\gamma-\delta)}{\zeta_{q}(2-\alpha+\beta+\gamma-\delta)} \frac1{2\pi i} \int_{(-\frac14)}\frac{G(s)}{s}g_{-\alpha,-\beta,-\gamma,-\delta}(-s)\\
\times&  H_{-\beta,-\gamma}(-s)\mathcal{Y}_{k,h,-\beta,-\alpha,-\delta,-\gamma}(-s)\zeta_q(1-\alpha-\delta+2s)\zeta_{q}(1+\beta+\gamma-2s) \d s.
\end{align*}
From \eqref{eqdefmY}, we deduce the symmetry property
\[
\mathcal{Y}_{k,h,-\beta,-\alpha,-\delta,-\gamma}(-s)=\mathcal{Y}_{h,k,\abgd}(s).
\]
Furthermore, it is known that
\begin{equation*}
X_{\abgd}g_{-\alpha,-\beta,-\gamma,-\delta}(-s)H_{-\beta,-\gamma}(-s)=g_{\abgd}(s)H_{\alpha,\delta}(s),
\end{equation*}
which corresponds precisely to the last formula in \cite[p.39]{You11}.

Applying these two relations yields that $Q_{h,k}(\abgd)$ and $Q_{\ol{h,k}}(\badg)$ share identical integrands but are integrated along conjugate contours $c_s=\frac14$ and $c_s=-\frac14$ respectively. The lemma then follows immediately from Cauchy's theorem.
\end{proof}

 The arguments in Lemma~\ref{1stQ} can be similary applied to other pairs of $Q_{h,k}$ and $Q_{\ol{h,k}}$.  Thus that the total contribution of the sum of all of them equals
\begin{align}
\label{secondmain}
\begin{split}
 X_{\alpha,\gamma} & Z_{h,k,q}(\beta,-\gamma,\delta, -\alpha)  +X_{\beta,\gamma}Z_{h,k,q}(\alpha,-\gamma,\delta, -\beta)+ X_{\alpha,\delta}Z_{h,k,q}(\beta,-\delta,\gamma, -\alpha)  \\
 & +X_{\beta,\delta}Z_{h,k,q}(\alpha,-\delta,\gamma, -\beta) +O\left(h^{-\frac14}k^{\frac14}q^{-\frac12+\ve}\r).
\end{split}
\end{align}

\section{Error terms for balanced terms}\label{secE}
For the purpose of simplifying subsequent formulae, we uniformly set all shift parameters $\alpha,\beta,\gamma,\delta$ to $0$. All arguments can be straightforwardly extended to handle sufficiently small nonzero values $\alpha,\beta,\gamma,\delta$ without introducing additional complexity.

  We want to apply Lemma \ref{lemEV} to the auxiliary function
\[
\mV(m,r,l)=\left(\frac{kN}{h_2}\r)^{\frac12}F^{\pm}_{M,N}(m),
\]
 where we set $n_1=fl$ in the definitions \eqref{eqdefF+}, \eqref{eqdefF-} and \eqref{eqdefFpm}.
 This function inherits the type $W(m/M)$ characteristic from $F^{\pm}_{M,N}(m)$, while simultaneously being smooth in variables $r$ and $l$. It satisfies the uniform bound
\[
m^{j_1}r^{j_2}l^{j_3}\frac{\partial^{j_1+j_2+j_3}}{\partial m^{j_1}\partial r^{j_2}\partial l^{j_3}}\mV(m,r,l)\ll_{j_1,j_2,j_3}q^{(j_1+j_2+j_3)\ve}.
\]
Applying the above for the functions $\mathring{\mV}_\pm$ defined in Lemma \ref{lemEV}, we obtain the corresponding estimates
\[
x^{j_1}r^{j_2}l^{j_3}\frac{\partial^{j_1+j_2+j_3}}{\partial x^{j_1}\partial r^{j_2}\partial l^{j_3}}\mathring{\mV}_\pm(x,r,l)\ll_{j_1,j_2,j_3}q^{(j_1+j_2+j_3)\ve}M^{\frac12}.
\]

Since the treatments for $E^+_{h,k}(M,N)$ and $E^-_{h,k}(M,N)$ are completely parallel, we unify their notation by writing $E_{h,k}(M,N)$ while suppressing the sign distinction. The condition $(l,h_1d)=1$ in \eqref{eqBlast} decomposes into two independent conditions $(l,h_1a)=1$ and $(l,d_{a})=1$. We eliminate the latter using the M\"{o}bius formula
\[
\bm{1}_{(l,d_a)=1}=\sum_{c\mid (l,d_a)}\mu(c).
\]
After applying dyadic decomposition to localize the variable $n_1$ in \eqref{eqBlast} ($n_1\asymp N_1$ with $h_2N_1\ll N^{\frac12}$),
we apply Lemma \ref{lemEV} to \eqref{eqBlast} to obtain
\begin{equation}\label{eqEMNE}
E_{h,k}(M,N)\ll q^\ve(hk)^{-\frac12}\left(\frac {hM}{kN}\r)^{\frac12}\max_{(f,h_1,k_1,a,c)}(|E_+|+|E_-|),
\end{equation}
where
\begin{equation*}
E_\pm=\frac1{a^2c^2k_1}\sum_{r}\alpha_r\sum_{m}\beta_m\sum_{(l,h_1a)=1}\frac1{l^2}\Omega(r,m,l)S(\ol{h_1a^2}dr,\pm m; ck_1l),
\end{equation*}
with $\alpha_r\ll r^\ve,~ \beta_m\ll m^\ve$ and the parameters subject to the constraints
\[
f\le N_1,\quad h_1h_2= h,\quad k_1\mid k, \quad (hk,d)=1, \quad ac\mid d,\quad (a,c)=1.
\]
Here, we also applied dyadic decomposition to localize all variables $r$, $m$, and $l$, and $\Omega(r,m,l)$ is a product of $\mathring{\mV}_\pm(m,r,l)$ with the three $W$-functions: $W(m/M^*)$, $W(r/R)$, and $W(l/L)$.

We set
\[
Q=uv,\quad u=ck_1,\quad v=h_1a^2,\quad s=d,
\]
which implies the simplified bound
\[
M^*\ll \frac{(Lu)^2v}{h_1M}.
\]
The sums $E_\pm$ can be reformulated as
\begin{equation*}
E_\pm=\frac{h_1k_1}{u^2v}\sum_{r}\alpha_r\sum_{m}\beta_m\sum_{(l,v)=1}\frac1{l^2}\Omega(r,m,l)S(sr\ol{v},\pm m; lu),
\end{equation*}
where $\Omega(r,m,l)$ inherits smoothness properties from $\mathring{\mV}_\pm(m,r,l)$ and $W$-functions, having compact support in $[R,2R]\times[M^*,2M^*]\times[L,2L]$ with
\[
R\ll \frac{kN}{dh_2},\quad M^*\ll \frac{(Lack_1)^2}{M},\quad L=\frac{N_1}{cf}, \quad N_1\ll \frac{N^{\frac12}}{h_2}.
\]
Applying Theorem \ref{thKls1}, we obtain the estimation
\begin{equation*}
E_\pm\ll \frac{h_1k_1}{s^{\frac12}}\frac{(1+X^{-\vartheta})X}{1+X}\B(1+X+\frac{s^{\frac12}}{Q^{\frac12}}\B)^\frac12\B(1+X+\frac{(s,Q)^{\frac12}R}{Q^{\frac12}}\B)^\frac12 \B(1+X+\frac{(M^*)^{\frac12}}{Q^{\frac12}}\B),
\end{equation*}
where the auxiliary parameter $X$ satisfies
\begin{equation}
\label{Xest}
X=\frac{(sRM^*)^{\frac12}}{uv^{\frac12}L}\ll \left(\frac{kN}{hM}\r)^\frac12.
\end{equation}
 Note that we have the following relations.
\[
\frac{M^*}{Q}\ll\frac{L^2u}{h_1M}\ll \frac{N_1^2k_1}{h_1M}\ll\frac{kN}{hM},\quad \frac{(s,Q)}{Q}\ll \frac1{h_1k_1}.
\]
Substituting the estimation for $X$ given in \eqref{Xest} and simplifying, we derive that
\begin{equation*}
E_\pm\ll \frac{h_1k_1}{s^{\frac12}}\B(\frac{kN}{hM}+\frac{s}{h_1k_1}\B)^\frac14\B(\frac{kN}{hM}+\frac{k^2N^2}{dh_1k_1^2}\B)^\frac14 \B(\frac{kN}{hM}\B)^{\frac12}.
\end{equation*}
Applying this to \eqref{eqEMNE}, we obtain the final bound
\begin{equation*}
E_{h,k}(M,N)\ll q^\ve\left(\frac{hk}{d}\r)^{\frac12}\B(\frac {(kN)^{\frac12}}{(hM)^{\frac12}}+\frac{k^{\frac12}N^{\frac34}}{h^{\frac12}M^{\frac14}d^{\frac12}}+\frac{(dN)^{\frac14}}{h^{\frac12}M^{\frac14}} +\frac{N^{\frac12}}{(hd)^{\frac12}}\B).
\end{equation*}
This result, when substituted into \eqref{eqMA}, completes the proof of \eqref{eqEhk}.

%
\section{Unbalanced terms}\label{secub}
In this section, we preliminary treat the unbalanced terms $B^\pm_{h,k}(M,N)$, preparing for the subsequent proof of theorems in next section.
For simplicity, we again set all parameters $\abgd$ to $0$. The arguments can be carried out straightforwardly to handle sufficiently small nonzero values of $\abgd$ without additional cost.

Without loss of generality, we adopt the convention $hM\ll kN$ throughout this section. Given the identical treatment of $B^+_{h,k}(M,N)$ and $B^-_{h,k}(M,N)$, we focus exclusively on the former, denoting its contribution to $\mmM_{h,k}$ by $\mB_{h,k}(M,N)$. The corresponding contribution to $\mmM(\abgd)$ is given by
\begin{equation}\label{eqmBB}
\mB(H,K,M,N)=\frac{1}{\ssqrt{HK}}\mathop{\sum_{h\le H}\sum_{k\le K}}_{(hk,q)=1}\alpha_h\beta_k\mB_{h,k}(M,N),
\end{equation}
where $\alpha_h\ll h^\ve$, $\beta_k\ll k^\ve$ are some complex coefficients supported on dyadic intervals $[H,2H]$ and $[K,2K]$, respectively.

We employ the Poisson summation formula to transform the unbalanced terms $\mB_{h,k}(M,N)$ and $\mB(H,K,M,N)$ into bilinear forms of incomplete Kloosterman sums. To derive sharp estimates for them, we supplement Weil's bound with the following specialized estimate.
\begin{lemma}[{\cite[Theorem 2.4]{KSWX23}}]\label{lemDS}
Let $q$ be a positive integer and $\alpha_a, \beta_b$ be sequences of complex numbers. For any $A,B\ge1$, the estimate
\begin{equation*}
\sum_{a\le A}\alpha_a\B|\sum_{\substack{b\le B\\ (b,q)=1}} \beta_b e\B(\frac{ca\ol{b}}{q}\B) \B|\ll\vert\bm\alpha\vert_2|\vert \bm\beta\vert_\infty A^{\tfrac12}Bq^\ve\left(A^{-\frac{1}2}B^{-\frac{1}4}q^{\frac{1}4}+A^{-\frac12}+q^{-\frac12}+B^{-\frac{1}2}\r)
\end{equation*}
holds uniformly in $c$ with $(c,q)=1$.
\end{lemma}
This lemma provides a refined control over the bilinear interactions, which is essential for handling the unbalanced regime.

We now eliminate the functions $\omega$ and $V$ via Mellin transforms, then apply dyadic decomposition to localize the variables $n_1, n_2$ within the ranges $n_1\asymp N_1$ and $n_2\asymp N_2$, where $N_1N_2\asymp N$ with $N_1\le N_2$. This procedure effectively reduces the term $\mB_{h,k}(M,N)$ to the following structured form
\begin{align}\label{eqE+1}
\mB_{h,k}(M,N)=\frac1{\vp^*(q)} & \sum_{d\mid q}\vp(d)\mu\left(\frac qd\r)\frac{1}{\ssqrt{MN}}\\
&\times \sum_{(m,q)=1}\sum_{(n_1,q)=1}\tau(m) W\left(\frac {m}M\right) W\left(\frac {n_1}{N_1}\right)
\sum_{\substack{n_2\equiv \ol{kn_1}hm\ppmod d\\ (n_2,q_d)=1}}W\left(\frac {n_2}{N_2}\right)\notag
\end{align}
for some smooth function $W$.
After removing the condition $(n_2,q_d)=1$ using the M\"{o}bius function, we apply the Poisson summation formula to the sum over $n_2$, obtaining
\begin{align*}
\sum_{\substack{n_2\equiv \ol{kn_1}hm\ppmod d\\ (n_2,q_d)=1}}W\left(\frac {n_2}{N_2}\right)&=\sum_{g\mid q_d}\mu(g)\sum_{\substack{n_2\equiv \ol{kn_1}hm\ppmod d\\ n_2\equiv0 \ppmod g}}W\left(\frac {n_2}{N_2}\right)\\
&=\frac{N_2}{d}\sum_{g\mid q_d}\frac{\mu(g)}{g}\sum_{l}e\B(\frac{hlm\ol{kgn_1}}{d}\B)\wh{W}\left(\frac {lN_2}{gd}\right).
\end{align*}
The ``main term'' arises from the $l=0$ term, which yields
\[
\wh{W}(0)\frac{N_2}{d}\sum_{g\mid q_d}\frac{\mu(g)}{g}.
\]
Substituting this into \eqref{eqE+1}, we observe that the contribution vanishes due to the following identity
\[
\sum_{d\mid q}\frac{\vp(d)}{d}\mu\left(\frac qd\r)\sum_{g\mid q_d}\frac{\mu(g)}{g}=0.
\]
This identity is readily verified for prime powers $q=p^k$ ($k\ge1$) and extends to general $q$ by multiplicativity.
Consequently, the remaining contribution is captured by the $l\neq0$ terms, leading to
\begin{align}\label{eqEpmhk}
&\mB_{h,k}(M,N)\ll \frac{q^\ve}{\vp^*(q)}\sum_{d\mid q}\sum_{g\mid q_d}\frac{\vp(d)}{dg}\frac{N_2}{\sqrt{MN}}\\
&\quad\quad\times\B|\sum_{l\neq0}\sum_{(m,q)=1}\sum_{(n_1,q)=1} \tau(m) e\B(\frac{hlm\ol{kgn_1}}{d}\B) W\left(\frac {m}M\right) W\left(\frac {n_1}{N_1}\right)\wh{W}\left(\frac {lN_2}{gd}\right)\B|.\notag
\end{align}

\section{Proof of Theorems \ref{thmmain1}--\ref{thmmain}}\label{secproof12}

\subsection{Proof of Theorem \ref{thmmain1}}
We adopt the exponential parametrization
\[
h=q^\varsigma,\quad k=q^\varrho,\quad M=q^\mu,\quad  N=q^\nu,
\]
where the unbalance condition ensures $\mu+\varsigma$ and $\nu+\varrho$ differ significantly.

We deduce from \eqref{eqMMM}, \eqref{Mdecomp}, \eqref{eqEhk}, \eqref{MD}, \eqref{MDdual} and \eqref{secondmain},  that the desired relation in \eqref{Mhkasmp} holds, provided that we show for $\eta\in (0,\tfrac1{20}(1-6\varrho))$,
\begin{equation}\label{eqAeta1}
\mB_{h,k}(M,N)\ll q^{-\eta+\ve}.
\end{equation}

Without loss of generality, we may restrict to $(l,d)=1$ in \eqref{eqEpmhk}.
Applying Weil's bound to the sum over $n_1$ and estimating the remaining terms trivially yields
\begin{equation}\label{eqAMN1}
\mB_{h,k}(M,N)\ll q^\ve\sum_{d\mid q}\frac{d}{\vp^*(q)}\left(\frac{M}{N}\r)^{\frac12}\B(d^{\frac12}+\frac{N_1}{d}\B)
\ll q^{\frac12+\ve}\left(\frac{M}{N}\r)^{\frac12},
\end{equation}
where the final inequality follows from $N_1\ll N^{\frac12}\ll q^{1+\ve}$.

Applying Lemma \ref{lemDS} with parameters $a=lm$ and $b=n_1$  to \eqref{eqEpmhk}, we obtain that
\begin{equation*}
\mB_{h,k}(M,N)\ll q^\ve\sum_{d\mid q}\frac{d}{\vp^*(q)}\left(\frac{MN_1}{N_2}\r)^{\frac12}\B(\left(\frac{Md}{N_2}\r)^{-\frac12}N_1^{-\frac14}d^{\frac14}
+\left(\frac{Md}{N_2}\r)^{-\frac12}+d^{-\frac12}+N_1^{-\frac12}\B).
\end{equation*}
 After simplification, this reduces to
\begin{equation}\label{eqAMN2}
\mB_{h,k}(M,N)\ll q^\ve \B(\B(\frac{N_1}{q}\B)^{\frac14}+ \B(\frac{M}{N}\B)^{\frac12}N_1^{\frac12}\B).
\end{equation}

Define $N_1=q^{\nu_1}$. We deduce from \eqref{eqEhk}, \eqref{eqBHK} and \eqref{eqAMN1} that it suffices to establish \eqref{eqAeta1} for the parameters satisfying
\begin{equation}\label{eqrange}
2-2\eta\le \mu+\nu\le 2, \quad 1-2\varrho-4\eta\le \nu-\mu\le 1+2\eta,\quad \nu_1\le \tfrac12\nu.
\end{equation}
From \eqref{eqrange}, we immediately deduce that
\begin{equation}\label{equpnu1}
\nu_1\le \tfrac14((\mu+\nu)+(\nu-\mu))\le \tfrac34+\tfrac{1}2\eta.
\end{equation}
Applying \eqref{eqAMN2} yields
\begin{equation*}
\mB_{h,k}(M,N)\ll q^{-\frac14(1-\nu_1)+\ve}+q^{-\frac12(\nu-\mu)+\frac12\nu_1},
\end{equation*}
while by \eqref{eqrange} and \eqref{equpnu1}, we see that
\begin{align*}
&-\tfrac14(1-\nu_1)\le -\tfrac14(\tfrac14-\tfrac12\eta)\le -\eta\quad \text{for}\quad \eta\le \tfrac1{18},\\
&-\tfrac12(\nu-\mu)+\tfrac12\nu_1\le -\tfrac12(\nu-\mu)+\tfrac18\left((\mu+\nu)+(\nu-\mu)\r)\\
&\quad\quad\quad\quad\quad\quad\ \ \ \le -\tfrac18+\tfrac34\varrho+\tfrac32\eta\le -\eta
\quad \text{for}\quad \eta\le \tfrac1{20}(1-6\varrho).
\end{align*}
The above readily implies that validity of \eqref{eqAeta1} subject to the conditions given in \eqref{eqrange}.
This completes the proof of Theorem \ref{thmmain1}.

\subsection{Proof of Theorem \ref{thmmain}}

  We apply \eqref{eqmBB} and \eqref{eqAMN1} to obtain an initial estimate for $\mB(H,K,M,N)$ so that
\begin{equation}\label{eqAMN1'}
\mB(H,K,M,N)\ll q^{\ve}\B(\frac {qHKM}N\B)^{\frac12}.
\end{equation}
 To proceed further, we apply \eqref{eqEpmhk} to see that
\begin{align*}
&\mB(H,K,M,N)\ll \frac{q^\ve}{\vp^*(q)}\sum_{d\mid q}\sum_{g\mid q_d}\frac{\vp(d)}{gd}\frac{N_2}{\sqrt{HKMN}}\\
&\quad\quad\times\B|\mathop{\sum_{h\le H}\sum_{k\le K}}_{(ab,q)=1}\alpha_h\beta_k\sum_{l\neq0}\sum_{(m,q)=1}\sum_{n_1} \tau(m) e\B(\frac{hlm\ol{kgn_1}}{d}\B) W\left(\frac {m}M\right) W\left(\frac {n_1}{N_1}\right)\wh{W}\left(\frac {lN_2}{gd}\right)\B|.
\end{align*}
Applying Lemma \ref{lemDS} with $a=hlm$ and $b=kn_1$, we obtain
\begin{equation}\label{eqAMN2'}
\mB(H,K,M,N)
\ll q^\ve\B( \B(\frac{KN_1}{q}\B)^{\frac14}+\B(\frac{KN_1}{q}\B)^{\frac12}\B)+ \B(\frac{qHM}{KN}\B)^{\frac12}\B(\B(\frac{KN_1}{q}\B)^{\frac12}+\frac{KN_1}{q}\B).
\end{equation}

  We now define the exponential parameters
\[
H=q^\varsigma,\quad K=q^\varrho,\quad M=q^\mu,\quad  N=q^\nu,\quad N_1=q^{\nu_1}.
\]
 We infer from \eqref{eqEHK}, \eqref{eqBHK} and \eqref{eqAMN1'} that it is enough to establish the bound
\begin{equation}\label{eqAeta'}
\mB(H,K,M,N)\ll q^{-\eta+\ve}
\end{equation}
for $\eta\in (0,\tfrac1{20}(1-12\varrho-10\varsigma))$,
subject to the constraints
\begin{equation}\label{eqrange'}
2-\varrho-\varsigma-2\eta\le \mu+\nu\le 2, \quad 1-4\varrho-2\varsigma-4\eta\le \nu-\mu\le 1+\varrho+\varsigma+2\eta,\quad \nu_1\le \tfrac12\nu.
\end{equation}
 We first derive from \eqref{eqrange'} that
\begin{equation}\label{equpnu1'}
\nu_1\le \tfrac14((\mu+\nu)+(\nu-\mu))\le \tfrac34+\tfrac14\varrho+\tfrac14\varsigma+\tfrac{1}2\eta,
\end{equation}
which implies
\[
\frac{KN_1}{q}=q^{-(1-\nu_1-\varrho)}\ll 1,
\]
provided that $0<\eta<\frac12(1-5\varrho-\varsigma)$.
Substituting this into \eqref{eqAMN2'} yields
\begin{equation*}
\mB(H,K,M,N)\ll q^{-\frac14(1-\nu_1-\varrho)+\ve}+q^{-\frac12(\nu-\mu)+\frac12\nu_1+\frac12\varsigma}.
\end{equation*}
 We then apply \eqref{eqrange'} and \eqref{equpnu1'} to see that
\begin{align*}
&-\tfrac14(1-\nu_1-\varrho)\le -\tfrac14(\tfrac14-\tfrac54\varrho-\tfrac14\varsigma-\tfrac12\eta)\le -\eta\quad \text{for}\quad \eta\le \tfrac1{18}(1-5\varrho-\varsigma),\\
&-\tfrac12(\nu-\mu)+\tfrac12\nu_1+\tfrac12\varsigma\le -\tfrac38(\nu-\mu)+\tfrac18(\mu+\nu)+\tfrac12\varsigma\\
&\ \ \ \ \ \ \ \ \ \ \ \ \quad \quad \quad\quad\quad\quad \le -\tfrac18+\tfrac32\varrho+\tfrac54\varsigma+\tfrac32\eta\le -\eta
\quad \text{for}\quad \eta\le \tfrac1{20}(1-12\varrho-10\varsigma).
\end{align*}
The above readily implies that validity of \eqref{eqAeta'} subject to the conditions given in \eqref{eqrange'}. This completes the proof of Theorem \ref{thmmain}.

\section{Proof of Theorem \ref{thKls}}
\subsection{Automorphic preliminaries}
We present here some fundamental properties of automorphic forms, drawing primarily from \cite{BM15}, \cite{Iwa95}, and \cite{DI82}.
Let $\ma$ denote a cusp of $\Gamma_0(Q)$ with $\sigma_{\ma}$ being the corresponding scaling matrix. For a holomorphic modular form $f$ of level $Q$ and weight $k$, its Fourier expansion around $\ma$ is given by
\[
f(\sigma_{\ma}z)=\sum_{n\ge1}\rho_f(\ma,n)(4\pi n)^{\frac k2}e(nz).
\]
For a Maa{\ss}  form $f$ with spectral parameter $t$, the expansion takes the form
\[
f(\sigma_{\ma}z)=\sum_{n\neq0}\rho_f(\ma,n)W_{0,it}(4\pi|n|y)e(nx),
\]
where $W_{0,it}(y)=(y/\pi)^{\frac12}K_{it}(y/2)$ is the Whittaker function.
Associated to each cusp $\mc$ of $\Gamma_0(Q)$ is an Eisenstein series $E_{\mc}(\sigma_\ma z,s)$, whose Fourier expansion at $s=\frac12+it$ is
\[
E_{\mc}(\sigma_\ma z, \tfrac12+it)=\delta_{\ma,\mc}y^{\frac12+it}+\varphi_{\ma,\mc}(\tfrac12+it)y^{\frac12-it} +\sum_{n\neq0}\rho_{\ma,\mc}(n,t)W_{0,it}(4\pi|n|y)e(nx).
\]
For convenience, we write $\rho_f(n)$ and $\rho_{\mc}(n,t)$ when $\ma=\infty$ for the Fourier coefficients.

For a newform $f$, its normalized Hecke eigenvalues $\lambda(n)$ satisfy the scaling relation
\[
\lambda(n)\rho_f(1)=\sqrt{n}\rho_f(n),
\]
and the multiplicativity formula
\begin{equation}\label{eqlambdam}
\lambda(mn)=\sum_{d\mid(m,n)}\mu(d)\chi_0(d)\lambda\left(\frac md\r)\lambda\left(\frac nd\r),
\end{equation}
where $\chi_0$ denotes the trivial character modulo $Q$.
The Ramanujan--Petersson conjecture implies the pointwise bound
\[
|\lambda(n)|\le \tau(n),
\]
which was proven for holomorphic forms in Deligne's seminal work \cite{Del74}. In the Maa{\ss} case, the current best result is due to Kim-Sarnak \cite{Kim03}, which asserts that
\begin{equation}
\label{theta}
|\lambda(n)|\le \tau(n)n^{\theta} \ \ \ \text{with}\ \ \ \theta =\tfrac 7{64}.
\end{equation}
While the full conjecture remains open, its averaged version
\[
\sum_{n\le x}|\lambda(n)|^2\ll x^{1+\ve},
\]
is well-established and frequently serves as a substitute for the original conjecture in applications.

For non-newforms, the exact multiplicativity \eqref{eqlambdam} no longer holds for their Fourier coefficients, but effective averaged substitutes are available. For any complex sequence $a_n$, V. Blomer, G. Harcos, and P. Michel \cite[p. 80]{BHM07} established for Eisenstein series the inequality
\begin{equation}\label{eqmEisen}
\sum_{\mc}\B|\sum_{(n,q)=1}a_n\ssqrt{qn}\rho_{\mc}(qn,\kappa)\B|^2\le 9\tau(Q)^3\tau(q)^4\sum_{g\mid (q,Q)}\sum_{\mc}\B|\sum_{(n,q)=1}a_{n}\ssqrt{gn}~\rho_{\mc}(gn,\kappa)\B|^2.
\end{equation}
For holomorphic cusp forms of level $Q$ and  weight $k$, we employ the newform theory to construct special $L^2$-bases $\mathcal{B}_k(Q)$, comprising of forms $f|_d(z):=f^*(dz)$ where $f^*$ is a normalized newform of level $Q_1\mid \frac Qd$. The Fourier coefficients of forms in $\mathcal{B}_k(Q)$ exhibit approximate multiplicativity. A similar basis $\mathcal{B}(Q)$ exists for Maa{\ss} cusp forms of level $Q$. The explicit construction of these bases is detailed in \cite[Sec.3.1]{BHM07}.
More precisely \cite[p.74]{BHM07}, if  $f\in \mathcal{B}_k(Q)$ (or $\mathcal{B}(Q)$) and $(s,n)=1$, then
\begin{equation}\label{eqmul1}
\sqrt{sn}~\rho_f(sn)=\sum_{\delta\mid (\frac{s}{(s,Q)},Q)}\mu(\delta)\chi_0(\delta)\lambda_{f^*}\left(\frac{s}{\delta{(s,Q)}}\r)\left(\frac{(s,Q)n}{\delta}\r)^{\frac12}~\rho_f\left(\frac{(s,Q)n}{\delta}\r),
\end{equation}
where $f^*$ denotes the underlying newform.  When $(n,sQ)=1$, this simplifies to
\begin{equation}\label{eqmul2}
\sqrt{sn}~\rho_f(sn)=\lambda_{f^*}(n)\sqrt{s}~\rho_f(s).
\end{equation}
Furthermore \cite[(5.4)]{BM15}, if $f^*$ satisfies the Ramanujan conjecture, then
\begin{equation*}
\B|\sum_{(n,q)=1}a_n\ssqrt{qn}\rho_f(qn)\B|^2\le \tau(s)^2\sum_{g\mid (q,Q)}\B|\sum_{(n,q)=1}a_{n}\ssqrt{gn}~\rho_f(gn)\B|^2.
\end{equation*}
Throughout this paper, we fix $\mathcal{B}_k(Q)$ and $\mathcal{B}(Q)$ to be these specially constructed bases.
\subsection{Kuznetsov formula and spectral large sieve}\label{secKlo}

  Let $\ma$ and $\mb$ be two cusps of the Hecke congruence group $\Gamma_0(Q)$, and let $S_{\ma,\mb}(m,n;\gamma)$ denote the associated Kloosterman sum defined in \eqref{eqdefS}. For $Q=uv$ with $(u,v)=1$, consider $S_{\infty,1/u}(m,n;\gammaup)$ with the scaling matrix $\sigma_{1/u}=\begin{pmatrix}
     \ssqrt{v} & 0 \\
     u\ssqrt{v} & \frac1{\ssqrt{v}}
\end{pmatrix}$.
The Kloosterman sum
$S_{\infty,1/u}(m,n;\gammaup)$ is well-defined precisely when $\gammaup=lu\ssqrt{v}$ for some integer $l$ coprime to $v$. In this case
\begin{equation*}
S_{\infty,1/u}(m,n;\gammaup)=e\left(n\frac{\ol{u}}{v}\right)S(m\ol{v},n;lu),
\end{equation*}
where $\ol{u}$ satisfies $u\ol{u}\equiv1 \pmod{v}$ and $\ol{v}$ satisfies $v\ol{v}=1 \pmod {lu}$
(see \cite[formula (1.6)]{DI82}).

The following lemma presents the Kuznetsov trace formula, which can be found in \cite[Theorems 9.4, 9.5, 9.7]{Iwa95}.
\begin{lemma}[Kuznetsov formula]\label{lemKf}
Let $\phi(x)$ be a twice continuously differentiable function on $[0,\infty)$ satisfying $\phi(0)=0$ and $\phi^{(j)}(x)\ll(1+x)^{-2-\ve}$ for $j=0,1,2$. Then for any cusps $\ma$, $\mb$ of  $\Gamma=\Gamma_0(Q)$ and any positive integers $m,n$, we have
\begin{align*}
\sum_\gamma\frac1{\gammaup}S_{\ma,\mb}(m,n;\gammaup)\phi\left(\frac{4\pi\ssqrt{mn}}{\gammaup}\right)=&\sum_{2\le k\equiv0(\bmod 2)}\sum_{f\in \mathcal{B}_k(Q)}\Gamma(k)\tilde{\phi}(k)\ssqrt{mn}~\ol{\rho}_f(\ma,m)~\rho_f(\mb,n)\notag\\
&+\sum_{f\in\mathcal{B}(Q)}\hat{\phi}(\kappa_f)\frac{\ssqrt{mn}}{\cosh(\pi\kappa_f)}~\ol{\rho}_f(\ma,m)~\rho_f(\mb,n)\notag\\
&+\frac1{4\pi}\sum_{\mc}\int_{-\infty}^{\infty}\hat{\phi}(\kappa)\frac{\ssqrt{mn}}{\cosh(\pi\kappa)}~\ol{\rho}_{\ma,\mc}(m,\kappa)~ \rho_{\mb,\mc}(n,\kappa) \d\kappa, \notag
\end{align*}
and
\begin{align*}
\sum_\gamma\frac1{\gammaup}S_{\ma,\mb}(m,-n;\gammaup)\phi\left(\frac{4\pi\ssqrt{mn}}{\gammaup}\right)= &\sum_{f\in\mathcal{B}(Q)}\breve{\phi}(\kappa_f)\frac{\ssqrt{mn}}{\cosh(\pi\kappa_f)}~\ol{\rho}_f(\ma,m)~\rho_f(\mb,-n)\notag\\
+&\frac1{4\pi}\sum_{\mc}\int_{-\infty}^{\infty}\breve{\phi}(\kappa)\frac{\ssqrt{mn}}{\cosh(\pi\kappa)}~\ol{\rho}_{\ma,\mc}(m,\kappa) ~\rho_{\mb,\mc}(-n,\kappa) \d\kappa,\notag
\end{align*}
where the sum is over all positive real $\gammaup$ for which the Kloosterman sum is well-defined. The Bessel transforms are given by
\begin{align*}
&\tilde{\phi}(k)=4i^k\int_0^\infty\phi(x)J_{k-1}(x)\frac{\d x}x,\notag\\
&\hat{\phi}(\kappa)=2\pi i\int_0^\infty\phi(x)\frac{J_{2i\kappa}(x)-J_{-2i\kappa}(x)}{\sinh(\pi \kappa)}\frac{\d x}x,\notag\\
&\breve{\phi}(\kappa)=8\int_0^\infty\phi(x)\cosh(\pi\kappa)K_{2i\kappa}(x)\frac{\d x}x.\notag
\end{align*}
\end{lemma}

The spectral large sieve inequality, often employed in conjunction with the Kuznetsov formula, is established as follows (see \cite[Theorem 2]{DI82}).
\begin{lemma}[Spectral large sieve]\label{lemsls}
For $K, N\ge1$ and complex coefficients $(a_n)_{n\in[N,2N]}$, let $\ma$ be a cusp of $\Gamma_0(Q)$ equivalent to $\frac u w$ with $(u,v)=1$ and $w\mid Q$. Then, for all such cusp $\ma$, all three quantities
\[
\sum_{\substack{2\le k\le K\\ k~ \text{even}}}\Gamma(k)\sum_{f\in\mathcal{B}_k(Q)}\B|\sum_{n}a_n\ssqrt{n}~\rho_f(\ma,n)\B|^2,  \ \ \ \ \ \ \ \ \ \ \ \ \sum_{\substack{f\in\mathcal{B}(Q)\\ |\kappa_f|\le K}}\frac1{\cosh(\pi\kappa_f)}\B|\sum_n a_n\ssqrt{n}~\rho_f(\ma,\pm n)\B|^2,
\]
\[
\sum_{\mc}\int_{-K}^K\frac1{\cosh(\pi\kappa)}\B|\sum_n a_n\ssqrt{n}~\rho_{\ma,\mc}(\pm n,\kappa)\B|^2d\kappa,
\]
are bounded by
\[
(K^2+\muup(\ma)N^{1+\ve})\|a_n\|_2^2,
\]
where $\muup(\ma)=Q^{-1}$ when $\ma=\infty$ and $\muup(\ma)=\left(w,\frac Qw\right)Q^{-1}$ otherwise.
\end{lemma}

\subsection{Extension of the spectral large sieve}
To prove Theorem \ref{thKls}, we apply the Kuznetsov formula (Lemma \ref{lemKf}), which leads us to consider a spectral sum of the form
\begin{equation}\label{eqss}
\sum_{\substack{f\in\mathcal{B}(Q)\\ |\kappa_f|\le K}}\frac1{\cosh(\pi\kappa_f)}\B|\sum_{n}a_n\ssqrt{sn}~\rho_f(sn)\B|^2.
\end{equation}
A direct application of the spectral large sieve (Lemma \ref{lemsls}) yields the bound
\[
\ll (sN)^\ve\left(K^2+\frac{sN}{Q}\r)\|a_n\|_2^2,
\]
where the dependence on $s$ introduces a significant loss, especially when $sN$ becomes large.
Our main goal is to derive an improved bound that effectively mitigates the impact of the parameter $s$.

For both the holomorphic and the Eisenstein cases, the Ramanujan--Petersson conjecture is well-established. Consequently, we can obtain bounds for \eqref{eqss} by applying Lemma \ref{lemsls}, leveraging
the almost multiplicativity property of the Fourier coefficients and the Ramanujan--Petersson conjecture.
\begin{theorem}\label{lemrR}
For $Q,s\in\mathbb{N}$, $K, N\ge 1$ and complex coefficients $(a_n)_{n\in[N,2N]}$, both two quantities
\[
\sum_{\substack{2\le k\le K\\ k~ \text{even}}}\Gamma(k)\sum_{f\in\mathcal{B}_k(Q)}\B|\sum_{n}a_n\ssqrt{sn}~\rho_f(sn)\B|^2,  \ \ \ \sum_{\mc}\int_{-K}^K\frac1{\cosh(\pi\kappa)}\B|\sum_n a_n\ssqrt{sn}~\rho_{\mc}(\pm sn,\kappa)\B|^2\d\kappa
\]
are bounded by
\begin{equation}\label{eqrR0}
(sQN)^\ve\left(K^2+\frac{(s,Q)N}{Q}\r)N\|a_n\|_\infty.
\end{equation}
\end{theorem}
\begin{proof}
We first treat the holomorphic case, with the Eisenstein spectrum proceeding analogously. By performing the substitution $n\rightarrow dn'$ where $d=(n,s^\infty)$, we decompose the sum as
\begin{equation*}
\sum_{n}a_n\ssqrt{sn}\rho_f(sn)\ll \sum_{d\mid s^\infty}\B|\sum_{(n',ds)=1}a_{dn'}\ssqrt{dsn'}~\rho_f(dsn')\B|.
\end{equation*}
Applying the multiplicativity relation \eqref{eqmul1} with $g=(ds,Q)/\delta$ and invoking Deligne's bound yields
\[
\sum_{(n',ds)=1}a_{dn'}\ssqrt{dsn'}~\rho_f(dsn')\ll (sN)^\ve\sum_{g\mid(ds,Q)}\B|\sum_{(n',ds)=1}a_{dn'}\ssqrt{gn'}~\rho_f(gn')\B|.
\]
Employing the Cauchy-Schwarz inequality over the sum in $d$ and $g$, we deduce that
\begin{equation*}
\B|\sum_{n}a_n\ssqrt{sn}\rho_f(sn)\B|^2\ll  (sQN)^\ve\sum_{d\mid s^\infty}\sum_{g\mid(ds,Q)}\B|\sum_{(n',s)=1}a_{dn'}\ssqrt{gn'}~\rho_f(gn')\B|^2,
\end{equation*}
where the factor $(sQN)^\ve$ accounts for the number of terms $d$ and $g$.
Inserting this into the holomorphic spectrum and applying Lemma \ref{lemsls} with $n=gn'\le (s,Q)N$ establishes the bound \eqref{eqrR0}.
The Eisenstein case follows identically, replacing \eqref{eqmul1} with the corresponding estimate \eqref{eqmEisen}.
\end{proof}

For the Maa{\ss} case in which the Ramanujan--Petersson conjecture remains open, additional techniques are necessary to manage the $s$-dependence. When $a_n\le 1$, V. Blomer and D. Mili\'cevi\'c \cite[Theorem 13]{BM15} established the fundamental bound
\begin{equation}\label{eqrRBM}
\sum_{\substack{f\in\mathcal{B}(Q)\\ |\kappa_f|\le K}}\frac1{\cosh(\pi\kappa_f)}\B|\sum_{(n,sQ)=1}a_n\ssqrt{sn}~\rho_f(sn)\B|^2\ll (sQKN)^\ve(s,Q)\B(K+\frac{s^{\frac12}}{Q^{\frac12}}\B)\B(K+\frac{N}{Q^{\frac12}}\B)N.
\end{equation}
While this bound requires the sum over $(n,sQ)=1$ and depends on $s$ through a factor $(s,Q)$ that is typically small (e.g., for cusp forms on $SL_2(Z)$), it can become substantial large for the Hecke congruence groups of high level, potentially yielding weaker estimates than applying the Kim-Sarnak upper bound for the Hecke eigenvalue. Note that this factor $(s,Q)$ arises  from the coefficient estimate
\[
\ssqrt{n}\ \rho_f(n)\ll (nQ)^\ve n^{\theta }(Q,n)^{\frac12-\theta }|\ \rho_{f^*}(1)|.
\]
To overcome this limitation, we develop a new relationship in the following.
\begin{proposition}\label{prorho}
For $f\in\mathcal{B}(Q)$ or $\mathcal{B}_k(Q)$ with associated newform $f^*$, let $s,s'\in\mathbb{N}$ with $s'\mid \frac{s}{(s,Q)}$. Then we have
\[
\ssqrt{s}~\rho_f(s)\lambda_{f^*}(s')=\sum_{g\mid s'} \left(\frac{ss'}{g^2}\r)^{\frac12}\rho_f\left(\frac{ss'}{g^2}\r).
\]
\end{proposition}
\begin{proof}
By combining \eqref{eqmul1} with the M\"obius inversion of \eqref{eqlambdam},
we obtain the identity
\begin{align}
\label{eqdelta}
\ssqrt{s}~\rho_f(s)\lambda_{f^*}(s')&=\sum_{\delta\mid\left(\frac{s}{(s,Q)},Q\r)} \mu(\delta)\chi_0(\delta) \lambda_{f^*}\left(\frac{s}{\delta{(s,Q)}}\r)\lambda_{f^*}(s')
\left(\frac{(s,Q)}{\delta}\r)^{\frac12}~\rho_f\left(\frac{(s,Q)}{\delta}\r)\notag\\
&=\sum_{g\mid s'}\sum_{\delta\mid\left(\frac{s}{g(s,Q)},Q\r)} \mu(\delta)\chi_0(\delta) \lambda_{f^*}\left(\frac{ss'}{g^2\delta{(s,Q)}}\r)
\left(\frac{(s,Q)}{\delta}\r)^{\frac12}~\rho_f\left(\frac{(s,Q)}{\delta}\r).
\end{align}
Setting $s''=\frac{ss'}{g^2}$ for $g\mid s'\mid \frac{s}{(s,Q)}$, we establish the relations
\[
(s,Q)=((s,Q),Q)\le\left(\frac{s}{g},Q\r)\le \left(\frac{ss'}{g^2},Q\r)=(s'',Q),
\]
\[
(s'',Q)=\left(\frac{ss'}{g^2},Q\r)\le (ss',Q)\le \left(\frac{s^2}{(s,Q)},Q\r)=(s,Q).
\]
These indicate that $(s,Q)=(s'',Q)$, yielding
\[
\left(\frac{ss'}{g^2(s,Q)},Q\r)=\left(\frac{s''}{(s'',Q)},Q\r),
\]
while
\[
\left(\frac{s}{g(s,Q)},Q\r)\mid\left(\frac{ss'}{g^2(s,Q)},Q\r)\mid\left(\frac{s^2}{g^2(s,Q)^2},Q\r).
\]
These relations indicate that $\left(\frac{s}{g(s,Q)},Q\r)$ and $\left(\frac{s''}{(s'',Q)},Q\r)$ share the same distinct prime factors. This identical prime factor structure allows us to simplify the $\delta$-sum in \eqref{eqdelta} via \eqref{eqmul1}, resulting in
\[
\sum_{\delta\mid\left(\frac{s''}{(s'',Q)},Q\r)} \mu(\delta)\chi_0(\delta) \lambda_{f^*}\left(\frac{s''}{\delta{(s'',Q)}}\r)
\left(\frac{(s'',Q)}{\delta}\r)^{\frac12}~\rho_f\left(\frac{(s'',Q)}{\delta}\r)=\ssqrt{s''}\ \rho_f(s''),
\]
thereby completing the proof.
\end{proof}
We also need the following estimate \cite[Lemma 12]{BM15}, which is another application of the Kuznetsov formula.
\begin{lemma}\label{lemBM1512}
For $K\ge1$, $Q, n\in\mathbb{N}$, we have
\[
\sum_{\substack{f\in\mathcal{B}(Q)\\ |\kappa_f|\le K}}\frac1{\cosh(\pi\kappa_f)}\left|\ssqrt{n}~\rho_f(n)\right|^2\ll (Kn)^\ve\B(K^2+\frac{(Q,n)^{\frac12}n^{\frac12}}{Q}\B).
\]
\end{lemma}
Building upon Proposition \ref{prorho} and employing a refined application of the Cauchy-Schwarz inequality, we achieve two significant improvements. First, we substantially eliminate the restrictive $(s,Q)$ factor presented in \eqref{eqrRBM}. Second, we completely relax the previously essential coprimality condition $(n,sQ)=1$. This relaxation is particularly advantageous when analyzing the parameters $h$ and $k$ in Theorems \ref{thmmain1}--\ref{thmmain}, where the level $Q$ becomes dependent on $k$ and $h$.
The resulting bound, formulated in the following theorem, demonstrates both remarkable generality and simplicity, rendering it widely applicable.
\begin{theorem}\label{thmsls}
Let $Q,s\in\mathbb{N}$, $K, N\ge 1$, and let $(a_n)_{n\in[N,2N]}$ be complex coefficients. Then we have
\begin{equation*}
\sum_{\substack{f\in\mathcal{B}(Q)\\ |\kappa_f|\le K}}\frac1{\cosh(\pi\kappa_f)}\B|\sum_{n}a_n\ssqrt{sn}~\rho_f(sn)\B|^2\ll (sQKN)^\ve\B(K+\frac{s^{\frac12}}{Q^{\frac12}}\B)\B(K+\frac{(s,Q)^{\frac12}N}{Q^{\frac12}}\B)N\|a_n\|_\infty.
\end{equation*}
\end{theorem}
\begin{proof}
We begin by analyzing the case where $(n,sQ)=1$, defining
\[
\mG(s,N)=\sum_{\substack{f\in\mathcal{B}(Q)\\ |\kappa_f|\le K}}\frac1{\cosh(\pi\kappa_f)}\B|\sum_{(n,sQ)=1}a_n\ssqrt{sn}~\rho_f(sn)\B|^2.
\]
We first utilize the  almost multiplicativity from \eqref{eqmul2} to separate the variables $s$ and $n$ and then we apply \eqref{eqmul1} to $\sqrt{s}\rho_f(s)$, obtaining the decomposition
\[
\ssqrt{s}~\rho_f(s)=\sum_{\delta\mid \left(\frac{s}{(s,Q)},Q\r)}\mu(\delta)\chi_0(\delta)\lambda_{f^*}\left(\frac{s}{\delta{(s,Q)}}\r)
\left(\frac{(s,Q)}{\delta}\r)^{\frac12}~\rho_f\left(\frac{(s,Q)}{\delta}\r).
\]
This leads to the inequality
\begin{align*}
\mG(s,N)\le \sum_{\delta\mid \left(\frac{s}{(s,Q)},Q\r)}& \sum_{\substack{f\in\mathcal{B}(Q)\\ |\kappa_f|\le K}}\frac1{\cosh(\pi\kappa_f)}\B|\sum_{(n,sQ)=1}a_n\lambda_{f^*}(n)\B|^2\\
&\times\left|\ssqrt{s}~\rho_f(s) \lambda_{f^*}\left(\frac{s}{\delta{(s,Q)}}\r)
\left(\frac{(s,Q)}{\delta}\r)^{\frac12}~\rho_f\left(\frac{(s,Q)}{\delta}\r)\right|.
\end{align*}
Through application of the Cauchy-Schwarz inequality to the spectral sum, we derive
\[
\mG(s,N)\le \sum_{\delta\mid \left(\frac{s}{(s,Q)},Q\r)}\mG_1^{\frac12}\mG_2^{\frac12},
\]
where the components are defined as
\begin{equation*}
\mG_1= \sum_{\substack{f\in\mathcal{B}(Q)\\ |\kappa_f|\le K}}\frac1{\cosh(\pi\kappa_f)}\left|\ssqrt{s}~\rho_f(s)\lambda_{f^*}\left(\frac{s}{\delta{(s,Q)}}\r)\right|^2,
\end{equation*}
and
\begin{equation*}
\mG_2=\sum_{\substack{f\in\mathcal{B}(Q)\\ |\kappa_f|\le K}}\frac1{\cosh(\pi\kappa_f)}\B|\sum_{(n,sQ)=1}a_n\lambda_{f^*}(n)\B|^4\B|\left(\frac{(s,Q)}{\delta}\r)^{\frac12}~\rho_f\left(\frac{(s,Q)}{\delta}\r)\B|^2.
\end{equation*}
Proposition \ref{prorho} applied with
$s'=\frac{s}{\delta{(s,Q)}}$ yields
\[
\mG_1\ll s^\ve \sum_{g\mid s'}\sum_{\substack{f\in\mathcal{B}(Q)\\ |\kappa_f|\le K}}\frac1{\cosh(\pi\kappa_f)}\B|\left(\frac{ss'}{g^2}\r)^{\frac12}\rho_f\left(\frac{ss'}{g^2}\r)\B|^2.
\]
From Lemma \ref{lemBM1512} and the identity $\left(\frac{ss'}{g^2},Q\r)=(s,Q)$ established in Proposition \ref{prorho}, we obtain the bound
\[
\mG_1\ll (sK)^\ve \left(K^2+\frac{s}{Q}\r).
\]
For $\mG_2$, employing the M\"obius inversion in \eqref{eqlambdam} together with \eqref{eqmul2} gives
\begin{align*}
\mG_2&=\sum_{\substack{f\in\mathcal{B}(Q)\\ |\kappa_f|\le K}}\frac1{\cosh(\pi\kappa_f)}\B|\sum_{(n_1n_2,sQ)=1}a_{n_1}a_{n_2}\sum_{r\mid (n_1,n_2)}\lambda_{f^*}\left(\frac{n_1n_2}{r^2}\r)\left(\frac{(s,Q)}{\delta}\r)^{\frac12}~\rho_f\left(\frac{(s,Q)}{\delta}\r)\B|^2\\
&\ll\sum_{\substack{f\in\mathcal{B}(Q)\\ |\kappa_f|\le K}}\frac1{\cosh(\pi\kappa_f)}\B|\sum_{n\le N^2}b_n \left(n\frac{(s,Q)}{\delta}\r)^{\frac12}\rho_{f}\left(n\frac{(s,Q)}{\delta}\r)\B|^2,
\end{align*}
where the coefficients satisfy
\[
b_n=\sum_{(n_1n_2,sQ)=1}a_{n_1}a_{n_2}\sum_{\substack{r\mid (n_1,n_2)\\ n_1n_2=r^2n}}1\ll \|a_n\|_\infty^2\sum_{r\ll N/\sqrt{n}}\tau(n)\ll \frac{N^{1+\ve}}{n^{\frac12}}\|a_n\|_\infty^2.
\]
Applying the spectral large sieve (Lemma \ref{lemsls}) reveals that
\begin{equation*}
\mG_2\ll N^\ve\left(K^2+\frac{(s,Q)N^2}{Q}\r)N^2\|a_n\|_\infty^2.
\end{equation*}
Combining these estimates, we establish for $(n,sQ)=1$ the fundamental bound
\begin{equation}\label{eqmG}
\mG(s,N)\ll (sKN)^\ve \B(K+\frac{s^{\frac12}}{Q^{\frac12}}\B)\B(K+\frac{(s,Q)^{\frac12}N}{Q^{\frac12}}\B)N\|a_n\|_\infty.
\end{equation}

For the general case without $(n,sQ)=1$,
we employ the following decomposition
\begin{equation*}
\sum_{n}a_n\ssqrt{sn}~\rho_f(sn)=\sum_{d\mid (sQ)^{\infty}}\sum_{(n,sdQ)=1}a_{dn}\ssqrt{sdn}~\rho_f(sdn).
\end{equation*}
Given that the $d$-sum contains $\ll(sQN)^\ve$ terms, an application of the Cauchy-Schwarz inequality yields
\begin{align*}
\sum_{\substack{f\in\mathcal{B}(Q)\\ |\kappa_f|\le K}}\frac1{\cosh(\pi\kappa_f)}\B|\sum_{n}a_n\ssqrt{sn}~\rho_f(sn)\B|^2&\ll (sQN)^\ve\sum_{\substack{d\le 2N\\ d\mid (sQ)^{\infty}}}\mG\left(sd,\frac Nd\r)\\
\ll& (sQKN)^\ve\B(K+\frac{s^{\frac12}}{Q^{\frac12}}\B)\B(K+\frac{(s,Q)^{\frac12}N}{Q^{\frac12}}\B)N\|a_n\|_\infty,
\end{align*}
where we have invoked the bound \eqref{eqmG}.\footnote{Note that employing estimate \eqref{eqrRBM} directly would introduce an additional factor $Q^{\frac12}$ in the final bound.}
This completes the proof.
\end{proof}
\subsection{Completion of the proof}
 We want to apply the Kuznetsov's formula (Lemma \ref{lemKf}) to treat the left-hand side of \eqref{eqKls}. Before doing so, we review key estimates for the Bessel transforms from Lemma \ref{lemKf}. As established in \cite[Lemma 7.1]{DI82}, the following bounds hold for real $\kappa$.
\begin{subequations}
\begin{align}
&\hat{\phi}(i\kappa), ~\breve{\phi}(i\kappa)\ll \frac{1+X^{-2\kappa}}{1+X}\ \ \ \ \text{for}\ \ 0<\kappa<\tfrac12,\label{eqBKa}\\
&\tilde{\phi}(\kappa),~\hat{\phi}(\kappa),~\breve{\phi}(\kappa)\ll \frac{1+|\log X|}{1+X}\ \ \ \ \text{for all}\ \ \kappa, \label{eqBKb}\\
&\tilde{\phi}(\kappa),~\hat{\phi}(\kappa),~\breve{\phi}(\kappa)\ll \left(1+\kappa^{-\frac12}X\r)|\kappa|^{-\frac52}\ \ \ \ \text{for all}\ \ |\kappa|\ge1. \label{eqBKc}
\end{align}
\end{subequations}
We now focus on the situation where $m$ and $n$ in the Kloosterman sum share the same sign (the complementary case can be treated analogously). Applying Lemma \ref{lemKf} with $\ma=\infty$ and $\mb=1/u$, we bound the left-hand side of \eqref{eqKls} by decomposing it into three principal contributions corresponding to the spectral decomposition
\begin{align*}
\ll&\sum_{2\le k\equiv0(\bmod 2)}\sum_{f\in \mathcal{B}_k(Q)}\Gamma(k)|\tilde{\phi}(k)|\B|\sum_{m}a_m\ssqrt{sm}~\ol{\rho}_f(sm)\B|~\B|\sum_{n} b_n\ssqrt{n}~\rho_f(\mb,n)\B|\\
&+\sum_{f\in\mathcal{B}(Q)}\frac{|\hat{\phi}(\kappa_f)|}{\cosh(\pi\kappa_f)}~\B|\sum_{m}a_m\ssqrt{sm}~\ol{\rho}_f(sm)\B|~\B|\sum_{n} b_n\ssqrt{n}~\rho_f(\mb,n)\B|\\
&+\frac1{4\pi}\sum_{\mc}\int_{-\infty}^{\infty}\frac{|\hat{\phi}(\kappa)|}{\cosh(\pi\kappa)}~\B|\sum_{m}a_m\ssqrt{sm}~\ol{\rho}_{\mc}(sm,\kappa)\B|~\B|\sum_{n} b_n\ssqrt{n}~\rho_{\mb,\mc}(n,\kappa)\B| \d\kappa.
\end{align*}
These terms respectively account for the holomorphic, Maa{\ss}, and continuous spectrum contributions. We now focus our analysis on the Maa{\ss} spectrum (the other cases follow similarly), decomposing it as
\[
\Sigma_1+\sum_{K\gg 1+X}\Sigma_K,
\]
where $\Sigma_1$ aggregates terms with $\kappa_f\ll 1+X$ and $\Sigma_K$ sums over $\kappa_f\asymp K$ for $K\gg 1+X$. After applying the Cauchy-Schwarz inequality, we estimate the $n$-sum via Lemma \ref{lemsls} and the $m$ sum through Theorem \ref{thmsls}, yielding
\[
\Sigma_1\ll \frac{1+|\log X|+X^{-\vartheta_Q}}{1+X}\B(1+X+\frac{s^{\frac12}}{Q^{\frac12}}\B)^{\frac12} \B(1+X+\frac{(s,Q)^{\frac12}M}{Q^{\frac12}}\B)^{\frac12}\B(1+X+\frac{N^{\frac12}}{Q^{\frac12}}\B)M^{\frac12}\|a_m\|_\infty^{\frac12}\|b_n\|_2,
\]
where we employ \eqref{eqBKa} and \eqref{eqBKb} to bound the Bessel transforms, and
\[
\Sigma_K\ll \left(1+K^{-\frac12}X\r)K^{-\frac52}\B(K+\frac{s^{\frac12}}{Q^{\frac12}}\B)^{\frac12} \B(K+\frac{(s,Q)^{\frac12}M}{Q^{\frac12}}\B)^{\frac12}\B(K+\frac{N^{\frac12}}{Q^{\frac12}}\B)M^{\frac12}\|a_m\|_\infty^{\frac12}\|b_n\|_2
\]
via \eqref{eqBKc}.
Summing over $K$ recovers the bound in \eqref{eqKls}. The holomorphic and continuous spectrum are treated similarly, with Theorem \ref{thmsls} being replaced by Theorem \ref{lemrR} wherever appropriate. This completes the proof of Theorem \ref{thKls}. \newline

\noindent{\bf Acknowledgments.} P. Gao is supported in part by the NSFC grant 12471003, X. Wu by the NSFC Grant 12271135, Anhui Provincial Natural Science Foundation Grant 2508085J005 and the Fundamental Research Funds for the Central Universities Grant JZ2025HGTG0254, and L. Zhao by the FRG Grant PS71536 at the University of New South Wales.


\begin{thebibliography}{99}

%
%
\bibitem[BBLR16]{BBLR16}
S. Bettin, H. M. Bui, X. Li, and M. Radziwi\l{}\l{},
A quadratic divisor problem and moments of the Riemann zeta-function,
J. Eur. Math. Soc. (JEMS) 22 (2020), 3953--3980.

\bibitem[BFKMM17+a]{BFK+17a}
V. Blomer, \'E. Fouvry, E. Kowalski, Ph. Michel, and D. Mili\'cevi\'c,
On moments of twisted $L$-functions,
Amer. J. Math. 139 (2017), 707--768.

\bibitem[BFKMM17+b]{BFK+17b}
V. Blomer, \'E. Fouvry, E. Kowalski, Ph. Michel, and D. Mili\'cevi\'c,
Some applications of smooth bilinear forms with Kloosterman sums,
Tr. Mat. Inst. Steklova 296 (2017), 24--35 (Russian). English transl. in Proc. Steklov Inst. Math. 296 (2017), 18--29.

\bibitem[BFKMMS23]{BFKMMS23} V. Blomer, \'E. Fouvry, E. Kowalski, P. Michel, D. Mili\'cevi\'c, and W. Sawin, The second moment theory of families of {$L$}-functions---the
              case of twisted {H}ecke {$L$}-functions, Mem. Amer. Math. Soc. 282(2023), v+148.





\bibitem[BHKM20]{BHKM20}
V. Blomer, P. Humphries, R. Khan, and M. B. Milinovich,
Motohashi's fourth moment identity for non-archimedean test functions and applications,
Compos. Math. 156 (2020), 1004--1038.

\bibitem [BHM07]{BHM07}
V. Blomer, G. Harcos, and P. Michel,
A Burgess-like subconvex bound for twisted $L$-functions (with appendix 2 by Z. Mao),
Forum Math. 19 (2007), 61--105.

\bibitem[BM15]{BM15}
V. Blomer and D. Mili\'{c}evi\'{c},
The second moment of twisted modular $L$-functions,
Geom. Funct. Anal. 25 (2015), 453--516.

\bibitem[BPRZ20]{BPRZ20} H. M. Bui, K. Pratt,  N. Robles, and A.
              Zaharescu, Breaking the {$\frac12$}-barrier for the twisted second moment
              of {D}irichlet {$L$}-functions, Adv. Math. 370 (2020), 107175, 40pp.

\bibitem [CFKRS05]{CFK+05}
J. B. Conrey, D. W. Farmer, J. P. Keating, M. O. Rubinstein, and N. C. Snaith,
Integral moments of $L$-functions,
Proc. London Math. Soc. (3) 91 (2005), 33--104.

\bibitem[CS07]{CS07}
J. B. Conrey and N. C. Snaith,
Applications of the $L$-functions ratios conjectures,
Proc. Lond. Math. Soc. (3) 94 (2007), 594--646



\bibitem[Del74]{Del74}
P. Deligne,
La conjecure de Weil, I,
Inst. Hautes \'{E}tudes Sci. Publ. Math. 43 (1974), 273--307.



\bibitem[DI82]{DI82}
J. M. Deshouillers and H. Iwaniec,
Kloosterman sums and fourier coefficients of cusp forms,
Invent. Math. 70 (1982), 219--288.






\bibitem[GZ25+a]{GZ25+a} P. Gao and L. Zhao, Upper bounds for moments of {D}irichlet {$L$}-functions to a fixed modulus, Mathematika, to appear.

\bibitem[GZ25+b]{GZ25+b} P. Gao and L. Zhao, Twisted fourth moment of {D}irichlet {$L$}-functions to a fixed modulus, arXiv:2507.18186.



\bibitem[HB81]{HB81}
D. R. Heath-Brown,
The fourth power mean of Dirichlet's $L$-functions,
Analysis 1 (1981), 25--32.

\bibitem[Hou16]{Hough16} B. Hough, The angle of large values of {$L$}-functions, J. Number Theory 167 (2016), 353--393.

\bibitem[HM10]{HM10} C. P. Hughes and P. Y. Matthew,
The twisted fourth moment of the Riemann zeta function.
J. Reine Angew. Math 641 (2010), 203--236.


\bibitem[Iwa95]{Iwa95} H. Iwaniec, Introduction to the Spectral Theory of Automorphic Forms, Bibl. Rev. Mat. Iberoamericana, Revista Matem\'atica Iberoamericana, Madrid (1995).
%




%


\bibitem[Kim03]{Kim03}
H. H. Kim,
Functoriality for the exterior square of $GL_4$ and the symmetric fourth of $GL_2$,
J. Amer. Math. Soc. 16 (2003), no. 1, 139--183, with Appendix 1 by Dinakar Ramakrishnan and Appendix 2 by Kim and Peter Sarnak.

%
%

\bibitem[KSWX23]{KSWX23}
B. Kerr, I. Shparlinski, X. Wu, and P. Xi,
Bounds on bilinear forms with Kloosterman sums,
J. London Math. Soc. (2) 108 (2023), 578--621.
%

\bibitem [Liu24]{Liu24} D. Liu, Zeros and moments of {$L$}-functions and applications, Thesis (Ph.D.)--University of Illinois Urbana-Champaign (2024), 95pp.



\bibitem[Mun17]{Munsch17} M. Munsch, Shifted moments of {$L$}-functions and moments of theta
              functions, Mathematika 63 (2017), 196--212.

\bibitem[RS15]{RadSou15} M. Radziwi{\l \l} and K. Soundararajan,  Moments and distribution of central {$L$}-values of quadratic
              twists of elliptic curves, Invent. Math. 202 (2015), 1029--1068.



\bibitem[RS05]{RS05} Z. Rudnick and K. Soundararajan, Lower bounds for moments of {$L$}-functions, Proc. Natl. Acad. Sci. USA,
 102 (2005), 6837--6838.

\bibitem[Sel46]{Selberg46} A. Selberg, Contributions to the theory of {D}irichlet's {$L$}-functions, Skr. Norske Vid.-Akad. Oslo I (1946), 62pp.

\bibitem[Sou07]{Sou07}
K. Soundararajan,
The fourth moment of Dirichlet $L$-functions,
in Analytic Number Theory, Clay Math. Proc. 7, Amer. Math. Soc., Providence, RI, 2007, pp. 239--246.

\bibitem[Sou08]{Sou08}
K. Soundararajan,
Extreme values of zeta and $L$-functions,
Math. Ann. 342 (2008), 467--486.

\bibitem[Sza24]{Szabo24} B. Szab\'o, High moments of theta functions and character sums, Mathematika 70, (2024), Paper No. e12242, 37 pp.

\bibitem[Wu19]{Wu19}
X. Wu,
The twisted mean square and critical zeros of Dirichlet $L$-functions,
Math. Z. 293 (2019), 825--865.

\bibitem[Wu23+a]{Wu23}
X. Wu,
The fourth moment of Dirichlet $L$-functions at the central value,
Math. Ann. 387 (2023), 1199--1248.

\bibitem[Wu23+b]{Wu23+} X. Wu,
The fourth moment of Dirichlet $L$-functions along the critical line,
Forum Math. 35 (2023), 1347--1371.


\bibitem[You11]{You11}
M. P. Young,
The fourth moment of Dirichlet $L$-functions,
Ann. of Math. 173 (2011), 1--50.

\bibitem[Zac19]{Zach19} R. Zacharias, Mollification of the fourth moment of {D}irichlet
              {$L$}-functions, Acta Arith. 191 (2019), 201--257.





\end{thebibliography}

\end{document}